\documentclass[11pt]{amsart}
\usepackage{amsmath}
\usepackage{amsthm}
\usepackage{amssymb}
\usepackage{amsfonts,mathrsfs}
\usepackage{stmaryrd}
\usepackage{amsxtra}  
\usepackage{epsfig}
\usepackage{verbatim}
\usepackage[all]{xy}
\usepackage{enumerate}
\usepackage{enumitem}
\usepackage{url}
\usepackage{hyperref}

\usepackage{dsfont}
\usepackage{color}
\usepackage{mathtools}
\usepackage{tikz}
\usetikzlibrary{fpu}
\usetikzlibrary{matrix}
\usepackage{bm}
\usepackage{bbm}

\usepackage{float}
\restylefloat{table}
\usepackage{tabularray}
\usepackage{graphicx}

\usepackage{array}
\newcolumntype{P}[1]{>{\centering\arraybackslash}p{#1}}

\theoremstyle{plain}
\newtheorem{theorem}[equation]{Theorem}
\newtheorem{proposition}[equation]{Proposition}
\newtheorem{lemma}[equation]{Lemma}
\newtheorem{corollary}[equation]{Corollary}

\usepackage{bm}

\theoremstyle{remark}
\newtheorem{remark}[equation]{Remark}
\numberwithin{equation}{section}

\newcommand{\dee}{\partial}

\newcommand{\w}{\wedge}

\usepackage{enumerate}

\renewcommand{\hat}{\widehat}

\renewcommand{\bar}{\overline}
\usepackage{scalerel,stackengine}
\stackMath
\newcommand\reallywidehat[1]{%
\savestack{\tmpbox}{\stretchto{%
  \scaleto{%
    \scalerel*[\widthof{\ensuremath{#1}}]{\kern.1pt\mathchar"0362\kern.1pt}%
    {\rule{0ex}{\textheight}}%WIDTH-LIMITED CIRCUMFLEX
  }{\textheight}% 
}{2.4ex}}%
\stackon[-6.9pt]{#1}{\tmpbox}%
}
\newcommand{\dbar}{\bar \partial}

\newcommand{\im}{\text{Im}}

\def\norm#1{\left\Vert#1\right\Vert}
\def\Gammaf#1{\Gamma\left(#1\right)}

\newcommand{\ch}{{\mathcal H}}
\newcommand{\ci}{{\mathcal I}}

\newcommand{\cs}{{\mathcal S}}

\newcommand{\sL}{{\mathscr L}}

\newcommand{\C}{{\mathbb C}}

\newcommand{\R}{{\mathbb R}}

\newcommand{\Z}{{\mathbb Z}}

\begin{document}

\title[Leray polygamma inequalities]{The Leray transform: distinguished measures, symmetries and polygamma inequalities}
\author{Luke D. Edholm  \& Yonatan Shelah}
%\subjclass[2010]{32W05}
\begin{abstract}
New symmetries, norm computations and spectral information are obtained for the Leray transform on a class of unbounded hypersurfaces in $\C^2$.
Emphasis is placed on certain distinguished measures, with results on operator norm monotonicity established by proving new polygamma inequalities.
Classical techniques of Bernstein-Widder and Euler-Maclaurin play crucial roles in our analysis.
Underpinning this work is a projective geometric theory of duality, which manifests here in the form of Hölder invariance. 
\end{abstract}
%\date{\today}
\thanks{The first author was supported by FWF Grants DOI 10.55776/P36884 and DOI 10.55776/I4557. \\
\indent The second author was supported by ERC Advanced Grant 101053085\\
\indent{\em 2020 Mathematics Subject Classification:} 32A26 (Primary); 26D15; 33B15; 47A30; 65B15 (Secondary)}
\address{Department of Mathematics\\Universit\"at Wien, Vienna, Austria}
\email{luke.david.edholm@univie.ac.at}
\address{Department of Mathematics\\University of Ljubljana, Slovenia}
\email{jonathan.shelah@fmf.uni-lj.si; yonshel@umich.edu}

\maketitle

\section{Introduction}\label{S:Intro}

The Leray (or Cauchy-Leray) transform $\bm{L}$ is a higher dimensional analogue of the familiar Cauchy transform on planar domains.
Let $\Omega$ be a convex domain in $\C^n$ with $C^2$ boundary $b\Omega := \cs$ and defining function $\rho$.
Given a function $f$ defined on $\cs$, its Leray transform $\bm{L}f$ is the holomorphic function defined for $z \in \Omega$ by the following formula:
\begin{subequations}
\begin{align}
\bm{L}f(z) &= \int_{\zeta \in \cs} f(\zeta) \sL(z,\zeta), \label{E:Leray-transform}\\
\sL(z,\zeta) &= \frac{1}{(2\pi i)^n} \frac{\dee\rho(\zeta)\w (\dbar\dee\rho(\zeta))^{n-1}}{\langle \dee\rho(\zeta),\zeta-z \rangle^n}.
\label{E:Leray-kernel}
\end{align}
\end{subequations}
The $(n,n-1)$-form $\sL$ is called the {\em Leray kernel} and $\langle\cdot,\cdot\rangle$ denotes the ordinary {\em bilinear} pairing of $(1,0)$-forms and vectors.
The operator $\bm{L}$ reproduces holomorphic functions from their boundary values and generates holomorphic functions from more general boundary data. 
It is straightforward to check that formula \eqref{E:Leray-transform} is independent of the choice of defining function $\rho$ used in \eqref{E:Leray-kernel}; see \cite[Section IV.3.2]{Range86}.
Additional background information on this operator is given in Section \ref{S:Leray-transform}.

A body of recent work \cite{Bar16,BarEdh20,BarEdh22,BarLan09,Bol05,LanSte13,LanSte14,LanSte17c,LanSte19} has investigated the mapping and invariance properties of $\bm{L}$ in a variety of settings, with the construction of interesting holomorphic function spaces seen as an important application.
For instance, Barrett and Edholm \cite{BarEdh22} studied $\bm{L}$ on the family of unbounded hypersurfaces
\begin{equation}
M_\gamma = \{(\zeta_1,\zeta_2) \in \C^2: \im{(\zeta_2)} = |\zeta_1|^\gamma \}, \qquad \gamma>1,
\end{equation}
and used their findings to construct dual pairs of projectively invariant Hardy spaces.
In the present work, further analysis of the Leray transform on $M_\gamma$ is carried out and detailed analytic information is obtained.
{\em Throughout the paper, we always assume $\gamma>1$.}

\subsection{Distinguished measures}\label{SS:2measures}

Let us parameterize $M_\gamma$ by $(\zeta_1,\zeta_2) = (r e^{i\theta},s+i r^\gamma)$, where $r \ge 0, \,\, s \in \R, \,\, \theta \in [0,2\pi)$, and consider measures of the form 
$$
\mu_d = r^d\, dr \w d\theta \w ds, \qquad d \in \R.
$$
%The rotational symmetry of $M_\gamma$ in $\zeta_1$ yields an orthogonal decomposition of the Leray transform $\bm{L} = \bigoplus_{k=0}^\infty \bm{L}_k$; see Section \ref{SS:Leray-transform-background}.
%The behavior of $\bm{L}$ in the spaces $L^2(M_\gamma,\mu_d)$ can be understood explicitly, with surprising connections to projective geometry and special function theory.
%Many precise results were obtained in \cite{BarEdh22} for each $\bm{L}_k$; this is reviewed in Sections~\ref{SS:Known-results-on-Mgamma} and \ref{SS:Leray-transform-background}.
Two measures are of particular importance to higher-dimensional Cauchy-Leray theory:
the {\em pairing measure} $\sigma$ and the {\em preferred measure} $\nu$; see \cite{Bar16,BarEdh20,BarEdh22}.
On $M_\gamma$, they take the form
\begin{align}
&\sigma = r^{\gamma-1}\, dr \w d\theta \w ds, \qquad (d=\gamma-1), \label{E:pairing-measure} \\
&\nu = r^{\frac{\gamma+1}{3}}\, dr \w d\theta \w ds, \qquad \big(d = \tfrac{\gamma+1}{3}\big). \label{E:preferred-measure}
\end{align}

The pairing measure $\sigma$ arises as (a constant multiple of) the Leray-Levi measure corresponding to the natural defining function $\rho(\zeta) = |\zeta_1|^\gamma - \im(\zeta_2)$; see \eqref{E:Leray-Lev-measure}.
This measure also appears in an alternative description of the Leray transform, facilitating a natural bilinear pairing between projectively invariant dual Hardy spaces; see \cite[Proposition 6.18]{BarEdh22}.
The preferred measure $\nu$ admits a projective transformation law; see \eqref{E:pref-measure-trans-law}. 
This leads to a corresponding Leray transformation law in \eqref{E:Leray-trans-law}, and to the definition of the aforementioned Hardy spaces in \eqref{E:def-invariant-Hardy}.
These notions were introduced by Barrett \cite[Sections 7 and 8]{Bar16} and developed further in \cite{BarEdh20,BarEdh22}. 
See \cite[Section 2]{BarEdh22} for a detailed discussion in the $M_\gamma$ setting.

\begin{remark}
Note that in the Heisenberg case ($\gamma=2$), the pairing and preferred measures are equal and coincide with ordinary Lebesgue measure $\mu_1$ on the parameter space $\R^3$.
In fact, the Leray transform and the (orthogonal) Szeg\H{o} projection coincide on $L^2(M_2,\mu_1)$.
\hfill $\lozenge$
\end{remark}

In the present article, we show that the behavior of $\bm{L}$ on $M_\gamma$ can be understood rather explicitly.
In particular, its exact norm is computed in a number of settings:

\begin{theorem}\label{T:Intro-norm-comps}

The norm of $\bm{L}$ in $L^2(M_\gamma,\mu_d)$ is determined for many $\gamma$ and $d$, including:
\begin{itemize}

\item[$(a)$] The norm of the Leray transform on $L^2(M_2,\mu_d)$ is given by
\begin{equation*}
\norm{\bm{L}}_{L^2(M_2,\mu_d)} = 
\begin{cases}
\sqrt{ \frac{\pi}{2} (1-d)\sec{\big(\frac{d\pi}{2}\big)} },  &d \in (-1,1)\cup(1,3) \\
 1,  &d = 1.
\end{cases}
\end{equation*}
%(By Corollary \ref{C:Leray-boundedness} below, $\bm{L}$ is bounded on $L^2(M_2,\mu_d)$ if and only if $d \in (-1,3)$.)

\item[$(b)$] The norm of the Leray transform on $L^2(M_\gamma,\mu_1)$ is given by
\begin{equation*}
\norm{\bm{L}}_{L^2(M_\gamma,\mu_1)} = 
\begin{cases}
(\gamma-1)^{\frac{1}{\gamma}-1} \sqrt{ \frac{\pi}{4}(\gamma-2)\gamma \csc\big( \frac{2\pi}{\gamma} \big)}, & \gamma \in (1,2)\cup(2,\infty) \\
1, & \gamma = 2.
\end{cases}
\end{equation*}

\item[$(c)$] The norm of the Leray transform on $L^2(M_\gamma,\sigma)$ is given by
\begin{equation*}
\norm{\bm{L}}_{L^2(M_\gamma,\sigma)} = \frac{\gamma}{2\sqrt{\gamma-1}}.
\end{equation*}

\item[$(d)$] The norm of the Leray transform on $L^2(M_\gamma,\nu)$ is given by 
\begin{equation*}
\norm{\bm{L}}_{L^2(M_\gamma,\nu)} = \sqrt{\frac{\gamma}{2\sqrt{\gamma-1}}}.
\end{equation*}
\end{itemize}
\end{theorem}

While the proofs of the parts of Theorem \ref{T:Intro-norm-comps} employ different techniques that appear throughout the paper, they all depend critically on analysis of the Leray symbol function $J(d,\gamma,k)$ defined in \eqref{E:def-symbol-function} below.
The proof of part $(a)$ is found in Corollary \ref{C:Leray-norm-comp-gamma=2}.
Parts $(b)$ and $(c)$ are found in Section \ref{S:series-analysis}, where the Bernstein-Widder theorem and the notion of complete montonicity play important roles; see Corollaries \ref{C:Leray-norm-pairing} and \ref{C:Leray-norm-Lebesgue}, which themselves are ``endpoint cases" of the more general Theorem \ref{T:Leray-boundedness-from-Bernstein-Widder}.
The preferred measure result, part $(d)$, is the most difficult to prove.
Starting in Section~\ref{S:EM-pref-symbol-function}, we use the Euler-Maclaurin summation formula to carry out lengthy analysis, eventually transforming the problem into concrete questions about the positivity of certain explicit polynomials.
The calculation of the norm of $\bm{L}$ with respect to the preferred measure $\nu$ concretely answers an open question on a natural pairing of projective dual Hardy spaces on $M_\gamma$; see Proposition \ref{P:efficiency-of-pairing}.

To contextualize our results and methods, we provide a short summary of previously known facts about the Leray transform acting on $L^2(M_\gamma,\mu_d)$.

\subsection{The Leray symbol function}\label{SS:Known-results-on-Mgamma}

In \cite{BarEdh22}, the rotational symmetry in $\zeta_1$ yields an orthogonal decomposition of $\bm{L}$ on $L^2(M_\gamma,\mu_d)$ into {\em sub-Leray operators} $\bm{L}_k$; see Section \ref{SS:Leray-transform-background}:
\begin{equation*}
\bm{L} = \bigoplus_{k=0}^\infty \bm{L}_k.
\end{equation*}

We recall for non-negative $k \in \Z$, the {\em Leray symbol function} of $L^2(M_\gamma,\mu_d)$:
\begin{equation}\label{E:def-symbol-function}
J(d,\gamma,k) = \frac{\Gamma\big(\frac{2k+1+d}{\gamma}\big) \Gamma\big(2k+2 - \frac{2k+1+d}{\gamma}\big)}{\Gamma(k+1)^2} \left( \frac{\gamma}{2}\right)^{2k+2} (\gamma-1)^{-\left( 2k+2 - \frac{2k+1+d}{\gamma} \right)}.
\end{equation}
%The square-root of the Leray symbol function gives the norm of each $\bm{L}_k$:

\begin{proposition}[Barrett-Edholm \cite{BarEdh22}]\label{P:norm-Lk}
$\bm{L}_k$ is bounded on $L^2(M_\gamma,\mu_d)$ if and only if 
\begin{equation}\label{E:increasing-intervals}
d \in ( -2k-1 \,,\, (2k+2)(\gamma-1)+1 ) := \ci_k(\gamma).
\end{equation}
When $d$ is taken in this interval,
\begin{equation}\label{E:norm_Lk-sqrt}
\norm{\bm{L}_k}_{L^2(M_\gamma,\mu_d)} = \sqrt{J(d,\gamma,k)}.
\end{equation}
\end{proposition}

The above intervals are nested, i.e., $\ci_{k}(\gamma) \subset \ci_{k+1}(\gamma)$, and they exhaust the real line as $k \to \infty$. 
Thus, for any given $d$, at most a finite number of  $\bm{L}_k$ are unbounded on $L^2(M_\gamma,\mu_d)$.
Remarkably, the symbol function (and therefore the norm of the $\bm{L}_k$) stabilizes to the same value as $k \to \infty$, independent of the choice of $d$:
\begin{proposition}[Barrett-Edholm \cite{BarEdh22}]\label{P:HF-norm}
For any $d \in \R$, the following limit holds
\begin{equation*}
\lim_{k\to\infty} \norm{\bm{L}_k}_{L^2(M_\gamma,\mu_d)} = \sqrt{\frac{\gamma}{2\sqrt{\gamma-1}}}.
\end{equation*}
\end{proposition}
In Figure \ref{Im:plots-of-J(d,5,k)} below, the stabilization of the symbol function can be seen for several different $d$ values in the case $\gamma=5$.
There is an intriguing relationship between the limit in Proposition~\ref{P:HF-norm} and a projective geometric invariant on $M_\gamma$. 
This is related to the Leray essential norm conjecture; see \cite[Section 2.3]{BarEdh22}.

The range of $d$ for which the full operator $\bm{L}$ is bounded on $L^2(M_\gamma,\mu_d)$ can now be deduced from the previous two propositions:

\begin{corollary}[Barrett-Edholm \cite{BarEdh22}]\label{C:Leray-boundedness}
The Leray transform $\bm{L}$ is a bounded operator from $L^2(M_\gamma,\mu_d)$ to itself if and only if $d \in (-1,2\gamma-1) = \ci_0(\gamma)$.
\end{corollary}

\begin{remark}
Due to the prominence of measures $\sigma$ and $\nu$, we use alternate notation when the symbol function corresponds to $d=\gamma-1$ and $d=\tfrac{\gamma+1}{3}$, as indicated by \eqref{E:pairing-measure} and \eqref{E:preferred-measure}:
\begin{subequations}
\begin{align}
C_\sigma(\gamma,k) &:= J(\gamma-1,\gamma,k), \label{E:def-of-pairing-symbol} \\
C_\nu(\gamma,k) &:= J\big(\tfrac{\gamma+1}{3},\gamma,k\big). \label{E:def-of-preferred-symbol}
\end{align}
\end{subequations}
We call $C_\sigma$ and $C_\nu$ the {\em pairing symbol} and {\em preferred symbol} functions, respectively.
\hfill $\lozenge$
\end{remark}
 
\subsection{Symbol function monotonicity and symmetries}\label{SS:new-symetries}

The results in Section \ref{SS:Known-results-on-Mgamma} show that $\bm{L}$ is bounded on $L^2(M_\gamma,\mu_d)$ if and only if the sub-Leray operator $\bm{L}_0$ is bounded there.
But the norm of the full Leray transform is determined by maximizing the function 
$$
k \mapsto \sqrt{J(d,\gamma,k)} = \norm{\bm{L}_k}_{L^2(M_\gamma,\mu_d)}
$$
over the positive integers, which is often very difficult.
One way to approach this maximization problem is to investigate when the symbol function is monotone.

In \cite{BarEdh22} it is shown that for $\gamma \neq 2$, the pairing symbol function $k \mapsto C_\sigma(\gamma,k)$ (corresponding to $d = \gamma-1$) is strictly decreasing.
But this is not necessarily true for other values of $d$, as is seen in Figure \ref{Im:plots-of-J(d,5,k)}, which illustrates how the behavior of the symbol function can change as $d$ varies within the boundedness interval $\ci_0(\gamma) = (-1,2\gamma-1)$.

%%%%%%%%%%%%%%%%%%%%%%%%%%%%%%%%%%%%%%%%%%%%%%
\begin{figure}[htp]
\centering
\includegraphics[width=.3\textwidth]{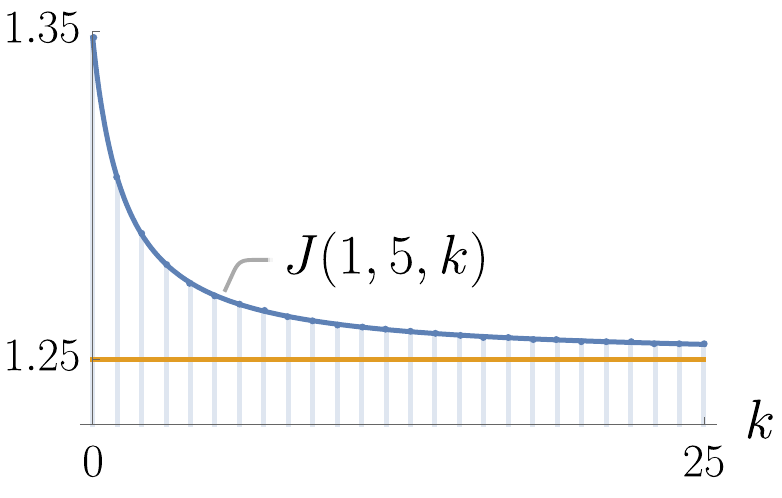}\quad
\includegraphics[width=.3\textwidth]{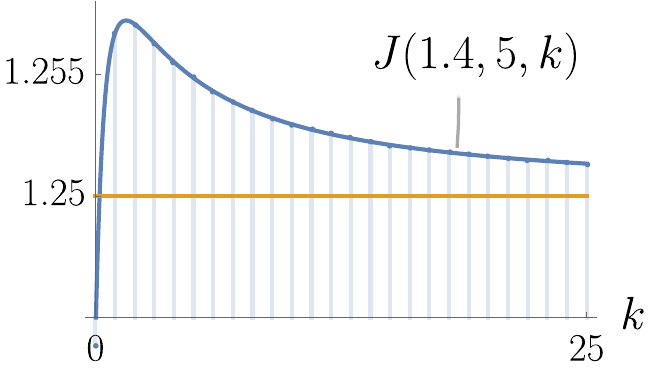}\quad
\includegraphics[width=.3\textwidth]{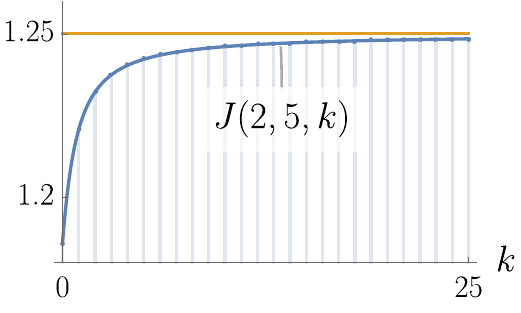}

\medskip

\includegraphics[width=.3\textwidth]{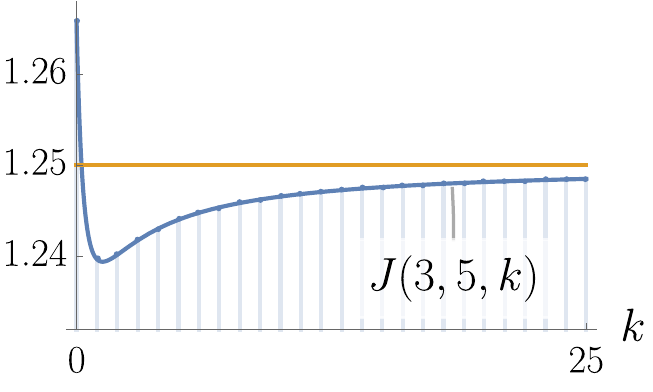}\quad
\includegraphics[width=.3\textwidth]{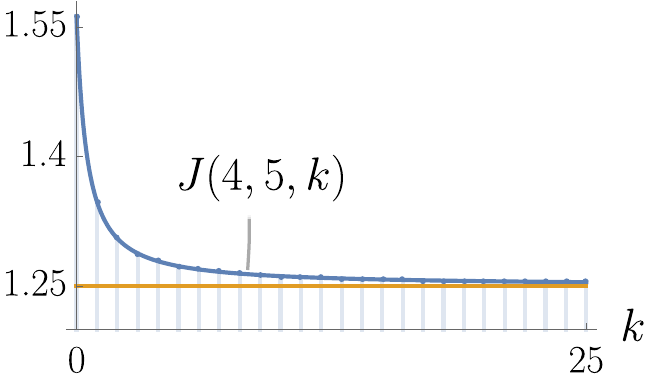}

\caption{Behavior of $k \mapsto J(d,5,k)$ for certain $1 \le d \le 4$.}
\label{Im:plots-of-J(d,5,k)}
\end{figure}
%%%%%%%%%%%%%%%%%%%%%%%%%%%%%%%%%%%%%%%%%%%%%%

While symbol function monotonicity fails to hold in general, for many choices of $d$ and $\gamma$ (including the cases that we are most interested in), it does hold:

\begin{theorem}\label{T:Lk-monotonicity-intro}
The sub-Leray operators $\bm{L}_k$ exhibit the following monotone behavior.
\begin{itemize}

\item[$(a)$] If $\gamma=2$ and $d \in (-1,1)\,\cup\,(1,3)$, the function $k \mapsto J(d,\gamma,k)$ is strictly decreasing on the non-negative integers.
Thus
\begin{equation*}
\norm{\bm{L}_k}_{L^2(M_2,\mu_d)} > \norm{\bm{L}_{k+1}}_{L^2(M_2,\mu_d)}.
\end{equation*}

\item[$(b)$] If $\gamma>2$ and $d \in (-1,1] \cup [\gamma-1,2\gamma-1)$, the function $k \mapsto J(d,\gamma,k)$ is strictly decreasing on the non-negative integers.
Thus
\[
\norm{\bm{L}_k}_{L^2(M_\gamma,\mu_{d})} > \norm{\bm{L}_{k+1}}_{L^2(M_\gamma,\mu_{d})}.
\]

\item[$(c)$] If $\gamma<2$ and $d \in (-1,\gamma-1] \cup [1,2\gamma-1)$, the function $k \mapsto J(d,\gamma,k)$ is strictly decreasing on the non-negative integers.
Thus
\[
\norm{\bm{L}_k}_{L^2(M_\gamma,\mu_{d})} > \norm{\bm{L}_{k+1}}_{L^2(M_\gamma,\mu_{d})}.
\]

\item[$(d)$] Let $\gamma \neq2$.
The preferred symbol function $k \mapsto C_\nu(\gamma,k) = J\big(\frac{\gamma+1}{3},\gamma,k \big)$ is strictly increasing on the non-negative integers.
Thus,
\begin{equation*}
\norm{\bm{L}_k}_{L^2(M_\gamma,\nu)} < \norm{\bm{L}_{k+1}}_{L^2(M_\gamma,\nu)}.
\end{equation*}

\end{itemize}
\end{theorem}

%Theorems \ref{T:Intro-norm-comps} and \ref{T:Lk-monotonicity-intro} are closely related.
The monotonicity result in part $(a)$ is shown in Theorem \ref{T:gamma=2-decreasing-symbol}, parts $(b)$ and $(c)$ are shown in Theorem \ref{T:Leray-boundedness-from-Bernstein-Widder}, and part $(d)$ in Theorem \ref{T:preferred-symbol-mono-increase}.
The proofs here depend on careful analysis of certain combinations of polygamma functions.
This is previewed in Section \ref{SS:Intro-polygamma} below.

\begin{remark}
The norm and monotonicity results in Theorems \ref{T:Intro-norm-comps} and \ref{T:Lk-monotonicity-intro} have immediate implications for the spectra of the related self-adjoint operators $\bm{L}^*\bm{L}$, $\bm{L}\bm{L}^*$ and the anti self-adjoint $\bm{L}^*-\bm{L}$; see \cite[Section 5.3]{BarEdh22}.
\hfill $\lozenge$
\end{remark}

\subsubsection{Hölder symmetry of the symbol}

In what follows, $\gamma^* = \tfrac{\gamma}{\gamma-1}$ denotes the Hölder conjugate of $\gamma$.
In \cite{BarEdh22} the pairing symbol function is shown to be Hölder symmetric, i.e., 
$$
C_\sigma(\gamma,k) = C_\sigma(\gamma^*,k).
$$
This observation meshes well with our understanding of Cauchy-Leray theory from a projective dual point of view, as $M_{\gamma^*}$ is known to be the projective dual hypersurface of $M_\gamma$; see \cite[Section 6.1]{BarEdh22}.
But it was never clarified in the previous work whether a more general version of this correspondence holds when other measures are considered, and in particular, if the preferred symbol function $C_\nu(\gamma,k)$ is Hölder invariant.

We prove here that a form of Hölder invariance holds for every choice of $\gamma \neq 2$, $d \in \R$.
Indeed, given a pair $(\gamma,d)$, we show that there is a unique $d'$ such that the behavior of $\bm{L}_k$ on $L^2(M_\gamma,\mu_d)$ parallels its behavior in $L^2(M_{\gamma^*},\mu_{d'})$.
To understand this symmetry and the value of $d'$ more explicitly, we re-parameterize the exponent of $\mu_d$ as follows.
Given $d \in \R$ (since $\gamma \neq 2$), choose the unique $a \in \R$ satisfying
\begin{equation}\label{E:delta_a_intro}
d = a(\gamma-2)+1 := \delta_{a}(\gamma).
\end{equation}

\begin{theorem}\label{T:Intro-Hölder-invariance}
The symbol function $J(\delta_a(\gamma),\gamma,k)$ is Hölder symmetric, in that
\[
J(\delta_{a}(\gamma),\gamma,k) = J(\delta_{a}(\gamma^*),\gamma^*,k).
\]
\end{theorem}
This theorem is proved in Section \ref{SS:Hölder-symmetry}.
Taking $a=\tfrac{1}{3}$ corresponds to $d = \tfrac{\gamma+1}{3}$, which implies the Hölder symmetry of the preferred symbol function: 
$$
C_\nu(\gamma,k) = C_\nu(\gamma^*,k).
$$
Taking $a=0$ shows that $J(1,\gamma,k) = J(1,\gamma^*,k)$, verifying the analogous invariance property holds for the symbol functions associated to the Lebesgue measure $\mu_1$.
The Hölder symmetry of the symbol function plays an important role in the proofs of both Theorems \ref{T:Intro-norm-comps} and \ref{T:Lk-monotonicity-intro}, and arises from the projective dual nature of the Leray transform; see Section \ref{SS:Proj-inv-duality}.

\subsection{Polygamma inequalities and complete monotonicity}\label{SS:Intro-polygamma}

Recall the {\em digamma function} $\psi$, which is the logarithmic derivative of the familiar $\Gamma$-function:
\begin{equation}\label{E:def-of-digamma}
\psi(r) = \frac{\Gamma'(r)}{\Gamma(r)}.
\end{equation}
Further derivatives of $\psi$ are called {\em polygamma functions}.
They satisfy interesting functional identities and arise naturally in both analytic and number theoretic settings; see \cite{AbrSteBook}.
The polygamma functions admit several practical representations; we make use of both the series form given in \eqref{E:polygamma-as-series} and the integral form given in \eqref{E:Polygamma-integral-formula}.

The following combination of polygamma functions plays a central role in this paper.
\[
\Phi(r,q) := r^2 \psi''(r+1-q) + 2r\psi'(r+1-q).
\]
Figure \ref{Im:plots-of-Phi(r,q)} shows the function $r \mapsto \Phi(r,q)$ plotted for certain fixed $q$.
In Lemma \ref{L:Ratio-into-series}, we show the behavior of the Leray symbol function is closely tied to the behavior of $\Phi(r,q)$: 
\begin{itemize}
\item if $\Phi(r,q) < 1$ for $r>q$, then $k \mapsto J(\delta_{1-q}(\gamma),\gamma,k)$ is strictly decreasing.
\item if $\Phi(r,q) > 1$ for $r>q$, then $k \mapsto J(\delta_{1-q}(\gamma),\gamma,k)$ is strictly increasing.
\end{itemize}
(Here $\delta_{1-q}(\gamma)$ is given by \eqref{E:delta_a_intro} with $a=1-q$.)

%%%%%%%%%%%%%%%%%%%%%%%%%%%%%%%%%%%%%%%%%%%%%%
%%%%%%%%%%%%%%%%%%%%%%%%%%%%%%%%%%%%%%%%%%%%%%
%%%%%%%%%%%%%%%%%%%%%%%%%%%%%%%%%%%%%%%%%%%%%%
\begin{figure}[htp]
\centering
\includegraphics[width=.3\textwidth]{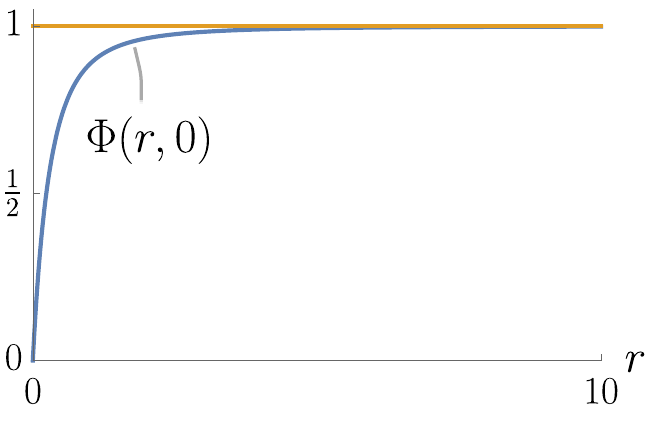}\quad
\includegraphics[width=.3\textwidth]{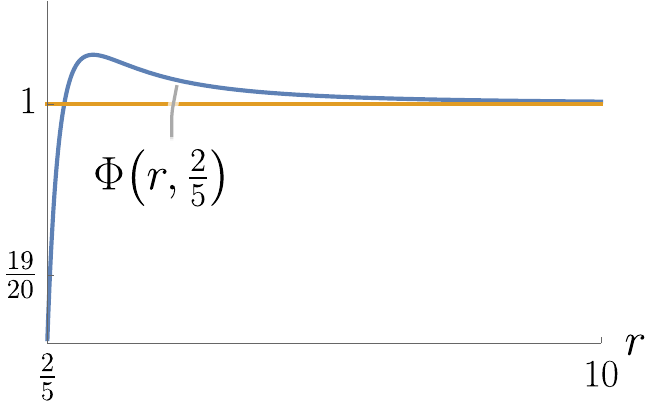}\quad
\includegraphics[width=.3\textwidth]{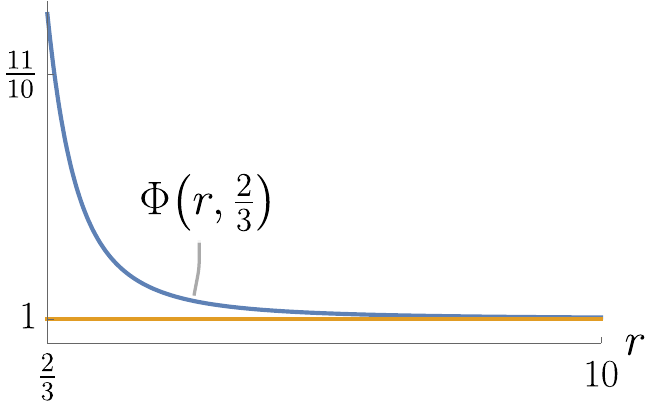}

\medskip

\includegraphics[width=.3\textwidth]{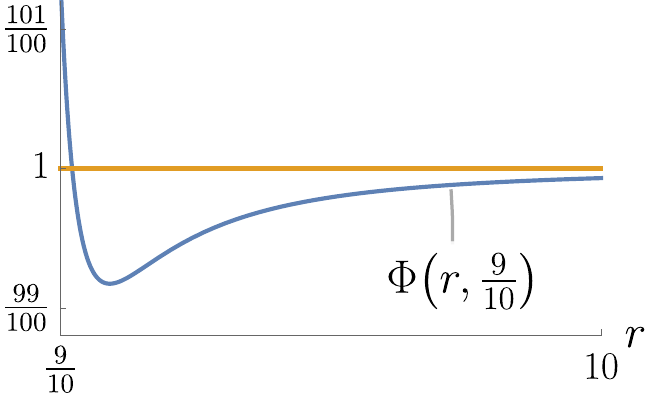}\quad
\includegraphics[width=.3\textwidth]{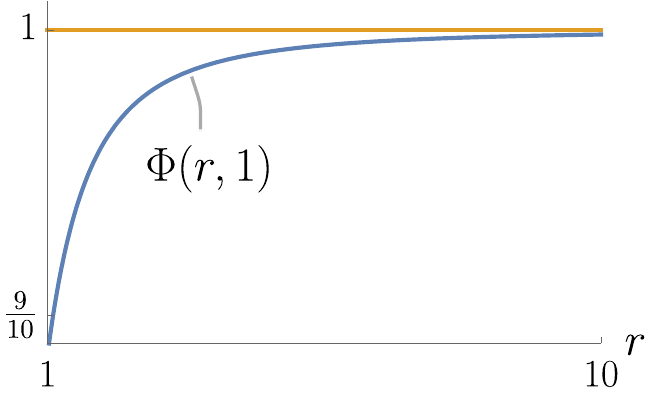}

\caption{Behavior of $r \mapsto \Phi(r,q)$ for certain $0 \le q \le 1$}
\label{Im:plots-of-Phi(r,q)}
\end{figure}
%%%%%%%%%%%%%%%%%%%%%%%%%%%%%%%%%%%%%%%%%%%%%%
%%%%%%%%%%%%%%%%%%%%%%%%%%%%%%%%%%%%%%%%%%%%%%
%%%%%%%%%%%%%%%%%%%%%%%%%%%%%%%%%%%%%%%%%%%%%%

\begin{theorem}\label{T:Intro-polygamma}
Sharp polygamma inequalities related to the symbol function are obtained.
\begin{itemize}
\item[$(a)$] If $q \in (-\infty,0]\cup[1,\infty)$, the following holds for $r>q$
\[
\Phi(r,q) = r^2 \psi''(r+1-q) + 2r\psi'(r+1-q) < 1.
\]

\item[$(b)$] When $q = \frac{2}{3}$, the following holds for $r > \tfrac{2}{3}$
\[
\Phi\big(r,\tfrac{2}{3} \big) = r^2 \psi''(r+\tfrac{1}{3}) + 2r\psi'(r+\tfrac{1}{3}) > 1.
\]
\end{itemize}
\end{theorem}

The proof of Theorem \ref{T:Intro-polygamma} part $(a)$ is found in Corollary \ref{C:Phi-estimate-q>=0}, where the inequality follows by proving the {\em complete monotonicity} of a closely related function in Theorem~\ref{T:Theta-et-al-is-completely-monotone-q>=1}.
The $q=\tfrac{2}{3}$ result in part $(b)$ corresponds directly to the preferred symbol function.
We prove this inequality in Theorem \ref{T:EdhShe-series} using a series representation for $\Phi(r,\tfrac{2}{3})$. 

There is a large body of literature on inequalities involving combinations of polygamma functions; see, e.g., \cite{Alzer1997,Alzer1998,GuoGuoQi_2010,GuoQi_2012_AAM,GuoQi_2013_PAMS,GuoQiZhao_2012,GuoQiSri_2012,Merkle_2005} and the references therein. 
The inequalities in Theorem \ref{T:Intro-polygamma} appear to be new and we suspect that mathematicians working in special function theory will take interest in results found in Sections \ref{S:series-analysis}, \ref{S:EM-pref-symbol-function}, and the Appendix.

The paper is organized as follows. 
Section \ref{S:Leray-transform} provides background material on the Leray transform.
In Section \ref{S:Symmetries-of-symbol} we establish new symmetries for the symbol function, which is then related to $\Phi(r,q)$.
In Section \ref{S:series-analysis} we introduce the Bernstein-Widder theorem and prove several polygamma and symbol function inequalities.
In Section \ref{S:EM-pref-symbol-function} the Euler-Maclaurin formula is used to study the preferred symbol function and the corresponding $\Phi(r,\tfrac{2}{3})$.
The proof of one especially difficult estimate used in Section \ref{S:EM-pref-symbol-function} is postponed until the Appendix.

\subsection{Acknowledgements}
We thank Dave Barrett and Bernhard Lamel for their comments.
We also thank the anonymous referee for positive feedback and several helpful suggestions to improve the paper.

%%%%%%%%%%%%%%%%%%%%%%%%%%%%%%%%%%%%%%%%%%%%%%%%%%%%%%
\section{The Leray transform}\label{S:Leray-transform}

The hypersurfaces $M_\gamma$ serve as the boundary of (unbounded) convex domains in $\C^2$.
But the Leray transform $\bm{L}$ and Leray kernel
\begin{equation*}
\sL(z,\zeta) = \frac{1}{(2\pi i)^n} \frac{\dee\rho(\zeta)\w (\dbar\dee\rho(\zeta))^{n-1}}{\langle \dee\rho(\zeta),\zeta-z \rangle^n}
\end{equation*}
can in fact be defined on the more general class of $\C$-convex domains.
A domain $\Omega$ is $\C$-convex if every (non-empty) intersection with a complex line is both connected and simply connected.
This condition ensures the non-vanishing of the denominator of the Leray kernel for any $z \in \Omega$.
Both the Leray transform and the notion of $\C$-convexity can be profitably studied in projective space $\C\mathbb{P}^n$. 
This is briefly discussed in Sections \ref{SS:Proj-inv-duality} and \ref{SS:proj-duality} below; see  \cite{AndPasSigBook04} for a detailed treatment of this topic.

It is often useful to decompose the Leray kernel $\sL$ into two pieces:
\begin{subequations}
\begin{align}
\ell(z,\zeta) &= \langle \dee\rho(\zeta),\zeta-z \rangle^{-n}, \label{E:Leray-coeff-function} \\
\lambda_\rho(\zeta) &= \frac{1}{(2\pi i)^n} \dee\rho(\zeta)\w (\dbar\dee\rho(\zeta))^{n-1}. \label{E:Leray-Lev-measure}
\end{align}
\end{subequations}
We refer to $\lambda_\rho$ as the {\em Leray-Levi measure}.
Readers are invited to draw the connection between $\lambda_\rho$ and the volume form $\theta \w d\theta^{n-1}$ associated to the pseudo-hermitian structure arising from the contact form $\theta = i \dee \rho$; see \cite{Lee1986}.
By itself, $\lambda_\rho$ clearly depends on the choice of defining function, though natural choices of $\rho$ often correspond to important measures.
See \cite{LanSte13} for a survey on $\bm{L}$ and related operators within the Cauchy-Fantappi\`e framework.

Barrett and Lanzani \cite{BarLan09} study the $L^2$-theory of $\bm{L}$ on smoothly bounded, strongly convex Reinhardt domains in $\C^2$. 
They obtain detailed spectral information on $\bm{L}^*\bm{L}$, $\bm{L}\bm{L}^*$ and $\bm{L}^*-\bm{L}$, and relate the essential norm of $\bm{L}$ to a geometric invariant of the domain.
Several articles by Lanzani and Stein have investigated other aspects the Leray transform.
In \cite{LanSte14} they prove that $\bm{L}$ preserves $L^p$ spaces ($1<p<\infty$) whenever the hypersurface $\cs$ is bounded, strongly $\C$-linearly convex and $C^{1,1}$ smooth.
In \cite{LanSte17c,LanSte19} they show that counter-examples to $L^p$-boundedness exist when the smoothness or convexity hypotheses are relaxed.

\subsection{Projective invariance}\label{SS:Proj-inv-duality}

In \cite[Section 5]{Bar16}, Barrett defines a projective geometric invariant associated to any smoothly bounded, strongly $\C$-convex hypersurface $\cs$.
He uses it \cite[Section~8]{Bar16} to construct a projectively invariant measure $\nu_\cs$:
If $\Phi$ is an automorphism of $\C\mathbb{P}^n$, the measure transforms as 
\begin{equation}\label{E:pref-measure-trans-law}
\Phi^*\big(\nu_{\Phi(\cs)}\big) = |\det \Phi'|^{\frac{2n}{n+1}}\,\nu_\cs.
\end{equation}

We refer to $\nu_\cs$ as the {\em preferred measure}.
Equation \eqref{E:pref-measure-trans-law} shows that the operator
\begin{equation}\label{E:isometry}
f \mapsto (\det \Phi')^{\frac{n}{n+1}}\cdot (f \circ \Phi)
\end{equation}
maps $L^2\big(\Phi(\cs),\nu_{\Phi(\cs)}\big)$ isometrically to $L^2(\cs,\nu_\cs)$.
If we use the Radon-Nikodym Theorem to re-express the Leray kernel \eqref{E:Leray-kernel} in terms of the preferred measure, then the Leray transform admits the transformation law
\begin{equation}\label{E:Leray-trans-law}
\bm{L}_\cs\left((\det \Phi')^{\frac{n}{n+1}} (f \circ \Phi) \right) = (\det \Phi')^{\frac{n}{n+1}} \left(\bm{L}_{\Phi(\cs)}(f) \circ \Phi \right);
\end{equation}
see \cite[Section 9]{Bar16} and \cite[Section 5]{Bol05}.
When the Leray transform is bounded on $L^2(\cs,\nu_\cs)$ (e.g., when Lanzani-Stein conditions of \cite{LanSte14} hold), it can be used to define projectively invariant Hardy spaces consisting of the boundary values of holomorphic functions:
\begin{equation}\label{E:def-invariant-Hardy}
H^2(\cs,\nu_\cs) = \bm{L}_\cs\left( L^2(\cs,\nu_\cs) \right).
\end{equation}

\subsection{Projective duality}\label{SS:proj-duality}

Given a smooth strongly $\C$-convex hypersurface $\cs \subset \C\mathbb{P}^n$, there is a unique complex tangent hyperplane at each $\zeta \in \cs$. 
This determines a unique point in a dual copy of $\C\mathbb{P}^n$; the set of all such points form the {\em dual hypersurface} $\cs^*$. 

In \cite[Section 6]{Bar16}, Barrett shows that $\cs^*$ is smoothly bounded and strongly $\C$-convex whenever $\cs$ is, and further, that $\cs$ and $\cs^*$ are diffeomorphic.
If $w:\cs \to \cs^*$ is a such a diffeomorphism, it can be used to pull back the space $H^2(\cs^*,\nu_{\cs^*})$ by setting
\begin{equation*}
H_{\sf dual}^2(\cs,\nu_{\cs}^*) = w^*\left( H^2(\cs^*,\nu_{\cs^*}) \right), \qquad \mathrm{where} \qquad \nu_\cs^* = w^*(\nu_{\cs^*}).
\end{equation*}

The {\em pairing measure} $\sigma_\cs$ is defined as the geometric mean of $\nu_\cs$ and $\nu_{\cs}^*$.

Given functions $f \in H^2(\cs,\nu_\cs)$ and $g \in H_{\sf{dual}}^2(\cs,\nu_{\cs}^*)$, define their {\em bilinear pairing} as
\begin{equation*}
\langle\!\langle f, g \rangle\!\rangle = \int_{\cs} f(\zeta)g(\zeta) \, \sigma_\cs(\zeta).
\end{equation*}
The map $\hat{\chi}_\gamma: g \mapsto \langle\!\langle \cdot, g \rangle\!\rangle$ gives a quasi-isometric identification of $H_{\sf{dual}}^2(\cs,\nu_{\cs}^*)$ with function theoretic dual space $H^2(\cs,\nu_\cs)'$.
In the $M_\gamma$ setting, we have the following result:

\begin{proposition}[Barrett-Edholm \cite{BarEdh22}]\label{P:efficiency-of-pairing}
The operator $\hat{\chi}_\gamma: H_{\sf{dual}}^2(M_\gamma,\nu) \mapsto H^2(M_\gamma,\nu)'$ is an invertible map with
\begin{equation*}
\norm{\hat{\chi}_\gamma^{-1}}_{H^2(M_\gamma,\nu)'} = \norm{\bm{L}}_{L^2(M_\gamma,\nu)}.
\end{equation*}
\end{proposition}

In light of Theorem \ref{T:Intro-norm-comps} part $(d)$ we can now express this as a concrete number:

\begin{theorem}
The operator $\hat{\chi}_\gamma: H_{\sf{dual}}^2(M_\gamma,\nu) \mapsto H^2(M_\gamma,\nu)'$ is an invertible map with
\begin{equation}\label{E:norm-of-chi-inv}
\norm{\hat{\chi}_\gamma^{-1}}_{H^2(M_\gamma,\nu)'} = \sqrt{\frac{\gamma}{2\sqrt{\gamma-1}}}.
\end{equation}
\end{theorem}

\subsection{Leray theory on $M_\gamma$}\label{SS:Leray-transform-background}

The following material is found in \cite[Sections 3, 5]{BarEdh22}.
As in the introduction, parameterize $M_\gamma$ by $(\zeta_1,\zeta_2) = (r e^{i\theta},s+i r^\gamma)$ and consider the measures
\begin{equation*}
\mu_d = r^d\, dr \w d\theta \w ds, \qquad d \in \R.
\end{equation*}
The pairing measure $\sigma$ ($d=\gamma-1$) arises as (a constant multiple of) the Leray-Levi measure corresponding to the defining function $\rho(\zeta) = |\zeta_1|^\gamma-\im{(\zeta_2)}$.

Rotational symmetry in $\zeta_1$ yields a decomposition of $L^2(M_\gamma,\mu_d)$ into an orthogonal sum of subspaces $L^2_k(M_\gamma,\mu_d)$, each consisting of functions of the form $f_k(r,s) e^{i k \theta}$.
Likewise, the Leray transform decomposes into orthogonal {\em sub-Leray} operators $\bm{L}_k$, each mapping $L^2(M_\gamma,\mu_d)$ to $L^2_k(M_\gamma,\mu_d)$ whenever it is bounded: $\bm{L} = \bigoplus_{k=0}^\infty \bm{L}_k$.
Given $f \in L^2(M_\gamma,\mu_d)$, the following inequality holds; see \cite[Theorem 5.13]{BarEdh22}:
\begin{equation}\label{E:Lk-upperbound-norm}
\norm{\bm{L}_k f}_{L^2(M_\gamma,\mu_d)} \le \sqrt{J(d,\gamma,k)} \cdot \norm{f}_{L^2(M_\gamma,\mu_d)},
\end{equation}
where $J(d,\gamma,k)$ is the symbol function defined in \eqref{E:def-symbol-function}.
It turns out there always exist functions $f_k(r,s)e^{ik\theta} \in L^2_k(M_\gamma,\mu_d)$ achieving equality in the relation \eqref{E:Lk-upperbound-norm} above.

If $f(r,s,\theta) = \sum_k f_k(r,s) e^{ik\theta}$, orthogonality implies the following; see \cite[Corollary 5.28]{BarEdh22}:
\begin{equation*}
\norm{\bm{L}f}^2 = \sum_{k=0}^\infty \norm{\bm{L}_k f}^2 = \sum_{k=0}^\infty \norm{\bm{L}_k f_k\,e^{ik\theta}}^2 \le \sum_{k=0}^\infty J(d,\gamma,k) \norm{f_k}^2 \le M \sum_{k=0}^\infty \norm{f_k}^2 = M \norm{f}^2,
\end{equation*}
where $M$ denotes the supremum of $J(d,\gamma,k)$ as $k$ ranges over the non-negative integers.
But since the norm of each $\bm{L}_k$ is achieved by a function in $L^2_k(M_\gamma,\mu_d)$, it follows that
\begin{equation}\label{E:Leray-norm-as-sup}
\norm{\bm{L}}_{L^2(M_\gamma,\mu_d)} = \sup\left\{ \sqrt{J(d,\gamma,k)} : k = 0,1,2,\dots \right\}.
\end{equation}

%%%%%%%%%%%%%%%%%%%%%%%%%%%%%%%%%%%%%%%%%%%%%%%%%%%%%%%%%%%%%%%%%%%%%%%%%%%%%%%

\section{Symmetries and Monotonicity}\label{S:Symmetries-of-symbol}

%We shall always assume $\gamma>1$ and let $\gamma^* := \frac{\gamma}{\gamma-1}$ denote the Hölder conjugate of $\gamma$.

\subsection{Hölder symmetry, $\gamma \neq 2$}\label{SS:Hölder-symmetry}

Proposition~\ref{P:norm-Lk} says $\bm{L}_k$ is bounded on $L^2(M_\gamma,\mu_d)$ if and only if $d\in \ci_k(\gamma)$.
Here, it is shown that for each $d$ there is a unique $d' \in \R$ such that the behavior of $\bm{L}_k$ in $L^2(M_\gamma,\mu_d)$ closely parallels its behavior in $L^2(M_{\gamma^*},\mu_{d'})$, where $\gamma^*$ is the Hölder conjugate of $\gamma$.
For $a \in \R$, define
\begin{equation}\label{E:def-of-delta}
\delta_a(\gamma) = a(\gamma-2)+1.
\end{equation}

\begin{theorem}\label{T:Holder-symmetry-symbol}
The symbol function $J(\delta_a(\gamma),\gamma,k)$ is Hölder symmetric, in that
\begin{equation*}
J(\delta_a(\gamma),\gamma,k) = J(\delta_a(\gamma^*),\gamma^*,k).
\end{equation*}
\end{theorem}
\begin{proof}
From \eqref{E:def-symbol-function} we have
\begin{equation*}
J(\delta_a(\gamma),\gamma,k) = \frac{\Gamma\big(\frac{2k+2+(\gamma-2)a}{\gamma} \big) \Gamma\big(2k+2 - \frac{2k+2+(\gamma-2)a}{\gamma} \big)}{\Gamma(k+1)^2} \left( \frac{\gamma}{2}\right)^{2k+2} (\gamma-1)^{\big( \frac{2k+2+(\gamma-2)a}{\gamma} -2k-2 \big)}.
\end{equation*}
But notice that
\begin{align}
\frac{2k+2+(\gamma^*-2)a}{\gamma^*} 
= \frac{2k+2-2a}{\gamma^*}+a 
&= \left(1-\frac{1}{\gamma}\right)(2k+2-2a) + a \notag \\
&= 2k+2-a - \frac{2k+2-2a}{\gamma} \notag \\
&= 2k+2 - \frac{2k+2+(\gamma-2)a}{\gamma}. \label{E:Holder-symm-1}
\end{align}
Thus, we have both
\begin{subequations}\label{E:Gamma-function-swap}
\begin{align}
&\Gamma\bigg(\frac{2k+2+(\gamma^*-2)a}{\gamma^*}\bigg) = \Gamma\bigg(2k+2 - \frac{2k+2+(\gamma-2)a}{\gamma} \bigg), \\
&\Gamma\bigg(\frac{2k+2+(\gamma-2)a}{\gamma}\bigg) = \Gamma\bigg(2k+2 - \frac{2k+2+(\gamma^*-2)a}{\gamma^*} \bigg).
\end{align}
\end{subequations}
Using that
\begin{equation*}
\Big( \frac{\gamma^*}{2} \Big)^{2k+2} = \Big( \frac{\gamma}{2} \Big)^{2k+2} (\gamma-1)^{-2k-2},
\end{equation*}
we deduce from \eqref{E:Holder-symm-1} that
\begin{align*}
(\gamma^*-1)^{\big( \frac{2k+2+(\gamma^*-2)a}{\gamma^*} -2k-2 \big)} 
&= (\gamma-1)^{\big( 2k +2 - \frac{2k+2+(\gamma^*-2)a}{\gamma^*} \big)} = (\gamma-1)^{\big( \frac{2k+2+(\gamma-2)a}{\gamma} \big)}.
\end{align*}
Combining the previous two equations we have
\begin{equation}\label{E:Holder-symm-2}
\Big( \frac{\gamma^*}{2} \Big)^{2k+2} (\gamma^*-1)^{\big( \frac{2k+2+(\gamma^*-2)a}{\gamma^*} -2k-2 \big)} 
= \Big( \frac{\gamma}{2} \Big)^{2k+2} (\gamma-1)^{\big( \frac{2k+2+(\gamma-2)a}{\gamma} -2k-2 \big)}.
\end{equation}
Putting together the equations in \eqref{E:Gamma-function-swap} with \eqref{E:Holder-symm-2} gives the result.
\end{proof}

\begin{remark}
The computations in the theorem above yield a Hölder symmetric reformulation of the symbol function whenever $\gamma \neq 2$:
\begin{equation*}
J(\delta_a(\gamma),\gamma,k) = \frac{\Gamma\big(\frac{2k+2+(\gamma-2)a}{\gamma} \big) \Gamma\big(\frac{2k+2+(\gamma^*-2)a}{\gamma^*} \big)}{\Gamma(k+1)^2} \left( \frac{\gamma}{2}\right)^{\big(\frac{2k+2+(\gamma-2)a}{\gamma} \big)} \Big( \frac{\gamma^*}{2}\Big)^{\big(\frac{2k+2+(\gamma^*-2)a}{\gamma^*} \big)}.
\end{equation*}
\hfill $\lozenge$
\end{remark}

\begin{corollary}\label{C:Hölder-symm-pairing-and-pref}
The pairing symbol function and the preferred symbol function are both Hölder symmetric, i.e., $C_\sigma(\gamma,k) = C_\sigma(\gamma^*,k)$ and $C_\nu(\gamma,k) = C_\nu(\gamma^*,k)$.
\end{corollary}
\begin{proof}
The pairing measure $\sigma=\mu_d$ with $d = \gamma - 1$. 
This $d$ value coincides with $\delta_a(\gamma)$ with $a=1$, so Theorem \ref{T:Holder-symmetry-symbol} applies.
The preferred measure $\nu=\mu_d$ with $d = \frac{\gamma+1}{3}$.
This $d$ value coincides with $\delta_a(\gamma)$ with $a=\frac{1}{3}$, so Theorem \ref{T:Holder-symmetry-symbol} again applies.
\end{proof}

\begin{remark}
In addition to $\sigma$ and $\nu$, Barrett and Edholm \cite[Section 6.2.2]{BarEdh22} also introduce the {\em dual preferred measure} $\nu^*$ on a $\C$-convex hypersurface $\cs$ as the pullback of the preferred measure on the dual hypersurface $\cs^*$. 
(Recall the discussion in Section \ref{SS:Proj-inv-duality} above, where these three measures are related by the property $\sqrt{\nu\nu^*} = \sigma$.)

When $\cs = M_\gamma$, 
\begin{equation*}
\nu^* = r^{\frac{5\gamma-7}{3}}\,dr \w d\theta \w ds,
\end{equation*}
so $\nu^* = \mu_d$ with $d=\tfrac{5\gamma-7}{3}$.
It follows from \eqref{E:increasing-intervals} that $\tfrac{5\gamma-7}{3} \in \ci_0(\gamma)$, meaning that $\bm{L}$ is bounded on $L^2(M_\gamma,\nu^*)$.
Since this $d$ value equals $\delta_a(\gamma)$ with $a = \frac{5}{3}$, Theorem \ref{T:Holder-symmetry-symbol} implies Hölder symmetry: if we set $C_{\nu^*}(\gamma,k) := J\big(\frac{5\gamma-7}{3},\gamma,k \big)$, then $C_{\nu^*}(\gamma,k) = C_{\nu^*}(\gamma^*,k)$.
\hfill $\lozenge$
\end{remark}

\begin{remark}
When $\gamma = 2$, the function $\delta_a(\gamma) = a(\gamma-2)+1 \equiv 1$ for every $a$.
Thus each pair of function spaces $L^2(M_\gamma,\mu_d)$ and $L^2(M_{\gamma^*},\mu_{d'})$ with parallel symbol function behavior reduces to a pairing of $L^2(M_2,\mu_1)$ with itself when $\gamma=2$.
We note that on this space, the Leray transform coincides with the Szeg\H{o} projection.
\hfill $\lozenge$
\end{remark}

\subsection{Heisenberg monotonicity}\label{SS:Additional-symmetries}
From \eqref{E:def-symbol-function}, observe that the symbol function greatly simplifies when $\gamma = 2$:
\begin{equation}\label{E:symbol-function_gamma=2}
J(d,2,k) = \frac{\Gamma\big(k+\frac{1+d}{2}\big) \Gamma\big(k + \frac{3-d}{2}\big)}{\Gamma(k+1)^2}.
\end{equation}
A fraction involving $\Gamma$-functions in this form is sometimes called Gurland's ratio; see \cite{Merkle_2005}.
\begin{proposition}\label{P:Lk-Heisenberg-symmetry}
For each non-negative integer $k$,
\begin{equation*}
\norm{\bm{L}_k}_{L^2(M_2,\mu_{d})} = \norm{\bm{L}_k}_{L^2(M_2,\mu_{2-d})}.
\end{equation*}
\end{proposition}
\begin{proof}
From \eqref{E:symbol-function_gamma=2} it is immediate that $J(d,2,k) = J(2-d,2,k)$.
Now use \eqref{E:norm_Lk-sqrt}.
\end{proof}

\begin{remark}\label{R:Heisenberg-Szego}
When $\gamma=2$ and $d=1$, we see that $\norm{\bm{L}_k}_{L^2(M_2,\mu_{1})} = 1$ for every $k$.
This of course is unsurprising, since the Leray transform is the Szeg\H{o} projection on $L^2(M_2,\mu_1)$.
\hfill $\lozenge$
\end{remark}

\begin{theorem}\label{T:gamma=2-decreasing-symbol}
Let $d$ be any real number.
The function $k \mapsto J(d,2,k)$ is decreasing whenever it is finite.
Consequently,
\begin{equation*}
\norm{\bm{L}_k}_{L^2(M_2,\mu_d)} \ge \norm{\bm{L}_{k+1}}_{L^2(M_2,\mu_d)}.
\end{equation*}
The inequality is strict for $d \neq 1$.
\end{theorem}
\begin{proof}
When $d=1$, Remark \ref{R:Heisenberg-Szego} says that the norm of $\bm{L}_k$ is equal to 1 for all $k$.

When $d \neq 1$, we see from \eqref{E:increasing-intervals} that $J(d,2,k)$ is finite if and only if 
\[
d \in \ci_k(2) = (-2k-1,2k+3).
\]
Since $\ci_{k}(2) \subset \ci_{k+1}(2)$ and these intervals exhaust the real line, it is clear that $J(d,2,\gamma)$ can be infinite for {\em at most} a finite number of integers $k$.

Observe that $d \in \ci_k(2)$ if and only if $2-d \in \ci_k(2)$. 
Proposition \ref{P:Lk-Heisenberg-symmetry} says it is sufficient to consider the case when $d>1$, which we now assume for the rest of the proof.

Taking logarithms, \eqref{E:symbol-function_gamma=2} becomes
\begin{equation*}
\log{J(d,2,k)} = \log{\Gamma\big(k + \tfrac{1+d}{2}\big)} + \log{\Gamma\big(k + \tfrac{3-d}{2}\big)} -2 \log{\Gamma(k + 1)}.
\end{equation*}
Now treat $k$ as a continuous variable and differentiate:
\begin{equation}\label{E:LogJ-step1}
\frac{\dee}{\dee k}\log{J(d,2,k)} = \psi(k + \tfrac{1+d}{2}) + \psi(k + \tfrac{3-d}{2}) - 2 \psi(k + 1),
\end{equation}
where $\psi$ is the digamma function defined in \eqref{E:def-of-digamma}.
Now consider the right-hand side of \eqref{E:LogJ-step1} in two separate pieces.
By the mean value theorem,
\begin{subequations}
\begin{equation}\label{E:MVT1}
\psi\big(k+\tfrac{1+d}{2}\big)-\psi(k+1) = \psi'(\xi_1)\big(\tfrac{d-1}{2}\big),
\end{equation}
for some $\xi_1 \in \big(k+1, k+ \tfrac{1+d}{2} \big)$.
Similarly,
\begin{equation}\label{E:MVT2}
\psi\big(k+\tfrac{3-d}{2}\big)-\psi(k+1) = -\psi'(\xi_2)\big(\tfrac{d-1}{2}\big),
\end{equation}
for some $\xi_2 \in \big(k+\tfrac{3-d}{2},k+1\big)$.
\end{subequations}

Combining \eqref{E:MVT1}, \eqref{E:MVT2} and \eqref{E:LogJ-step1} we use the mean value theorem again,
\begin{align*}
\frac{\dee}{\dee k}\log{J(d,2,k)} = \big(\tfrac{d-1}{2}\big)(\psi'(\xi_1)-\psi'(\xi_2)) = \big(\tfrac{d-1}{2}\big) \psi''(\xi_3)(\xi_1-\xi_2),
\end{align*}
for some $\xi_3 \in (\xi_2,\xi_1)$.
It is well-known that $\psi''<0$ on the positive reals (see, e.g., formula \eqref{E:polygamma-as-series} below).
This means that $k \mapsto \log{J(d,2,k)}$ is strictly decreasing for $k \in [0,\infty)$, which implies that $k \mapsto J(d,2,k)$ is strictly decreasing for $k \in [0,\infty)$.
Now restrict $k$ to the non-negative integers and recall \eqref{E:norm_Lk-sqrt} to complete the proof. 
\end{proof}

By Corollary \ref{C:Leray-boundedness}, the Leray transform is bounded on $L^2(M_2,\mu_d)$ if and only if $d \in \ci_0(2) = (-1,3)$.
When this holds, the norm of the full Leray transform can be computed:
\begin{corollary}\label{C:Leray-norm-comp-gamma=2}
Let $d \in (-1,3)$. 
The norm of the Leray transform on $L^2(M_2,\mu_d)$ is
\begin{equation*}
\norm{\bm{L}}_{L^2(M_2,\mu_d)} = 
\begin{cases}
\sqrt{ \frac{\pi}{2} (1-d)\sec{\big(\frac{d\pi}{2}\big)} },  &d \in (-1,1)\cup(1,3) \\
 1,  &d = 1.
\end{cases}
\end{equation*}
This formula is continuous at $d=1$.
\end{corollary}
\begin{proof}
The $d=1$ case can be seen directly from \eqref{E:symbol-function_gamma=2} (also see Remark \ref{R:Heisenberg-Szego}) and L'Hôpital's rule shows continuity at $d=1$.
When $d \neq 1$,  Theorem \ref{T:gamma=2-decreasing-symbol} says $k \mapsto J(d,\gamma,k)$ is strictly decreasing in $k$. 
Thus by \eqref{E:Leray-norm-as-sup},
\[
\norm{\bm{L}}_{L^2(M_2,\mu_d)} = \sqrt{J(d,2,0)}.
\]
Now compute:
\begin{equation*}
J(d,2,0) = \big(\tfrac{1-d}{2}\big) \Gamma\big(1-\tfrac{1-d}{2} \big)\Gamma\big(\tfrac{1-d}{2}\big) = \big(\tfrac{1-d}{2}\big)\pi \csc\big(\tfrac{\pi}{2}-\tfrac{d\pi}{2}\big)
= \tfrac{\pi}{2} (1-d)\sec\big(\tfrac{d\pi}{2}\big).
\end{equation*}
In this computation we have used the factorial property of the $\Gamma$-function along with Euler's reflection formula: $\Gamma(z)\Gamma(1-z) = \pi \csc(\pi z)$.
\end{proof}

\subsection{Polygamma functions and ratios of successive symbols}

In this section, once again consider $\gamma \neq 2$ and set $\delta_a(\gamma) = a(\gamma-2) + 1$.

The polygamma functions are the successive derivatives of the digamma function $\psi$, defined in \eqref{E:def-of-digamma}.
For integers $m \ge 1$, these functions admit the following series representations; see \cite[Equation 6.4.10]{AbrSteBook}:
\begin{equation}\label{E:polygamma-as-series}
\psi^{(m)}(r) = (-1)^{m+1} m! \sum_{j=1}^\infty \frac{1}{(r+j-1)^{m+1}}.
\end{equation}

The following polygamma combination is closely tied to Leray symbol function and is prominently featured throughout the paper (recall Figure \ref{Im:plots-of-Phi(r,q)} in Section \ref{S:Intro}):
\begin{subequations}
\begin{equation}\label{E:polygamma-combo}
\Phi(r,q) = 2r \psi'(r+1-q)+r^2 \psi''(r+1-q).
\end{equation}
In light of \eqref{E:polygamma-as-series}, it is often useful to write this function as a series
\begin{equation}\label{E:Important-series}
\Phi(r,q) = \sum_{j=1}^\infty \frac{2r(j-q)}{(r+j-q)^3}.
\end{equation}
\end{subequations}

The establishment of new inequalities related to the polygamma functions has been an active area of research for decades; see \cite{Alzer1997,Alzer1998,GuoGuoQi_2010,GuoQi_2012_AAM,GuoQi_2013_PAMS,GuoQiZhao_2012,GuoQiSri_2012,Merkle_2005} for results with similar flavor to those we prove below.
A typical way in which new polygamma inequalities are obtained is to show that some auxiliary function is {\em completely monotone} (see Section \ref{SS:Berstein-Widder}), a checkable condition thanks to the celebrated Bernstein-Widder theorem (see Theorem \ref{T:Bernstein-Widder}).

In Sections \ref{S:series-analysis} and \ref{S:EM-pref-symbol-function}, we prove inequalities involving $\Phi(r,q)$ and related functions, both by use of Bernstein-Widder and in situations where it fails to apply.
Our inequalities appear to be new and the methods are likely adaptable to more general settings.

For later convenience, we re-frame the finiteness of the symbol function in terms of $\delta_{1-q}$.
(Setting $a=1-q$ in \eqref{E:def-of-delta} yields a symmetric formulation of the following result):

\begin{lemma}\label{L:Control-on-q-for-boundedness}
Let $\gamma\neq 2$. The symbol function $J(\delta_{1-q}(\gamma),\gamma,k)$ is finite for all non-negative integers $k$ if and only if 
\begin{equation*}
|q| < \frac{\gamma}{|\gamma-2|}.
\end{equation*}
\end{lemma}
\begin{proof}
By Proposition \ref{P:norm-Lk}, $J(\delta_{1-q}(\gamma),\gamma,k)$ is finite for all $k$ if and only if
\[
\delta_{1-q}(\gamma) \in \bigcap_{k=0}^\infty \ci_k(\gamma) = \ci_0(\gamma) = (-1,2\gamma-1).
\]
For $\gamma>2$ this means
\begin{align*}
-1 < (1-q)(\gamma-2) + 1 < 2\gamma-1 \qquad &\Longleftrightarrow \qquad -\frac{\gamma}{\gamma-2} < q < \frac{\gamma}{\gamma-2},
\end{align*}
confirming the inequality. When $\gamma<2$,
\begin{align*}
-1 < (1-q)(\gamma-2) + 1 < 2\gamma-1 \qquad &\Longleftrightarrow \qquad -\frac{\gamma}{2-\gamma} < q < \frac{\gamma}{2-\gamma},
\end{align*}
confirming the inequality.
\end{proof}
\begin{remark}
Observe that the estimate appearing in Lemma \ref{L:Control-on-q-for-boundedness} is Hölder symmetric, i.e., invariant under the change of variable $\gamma \mapsto \gamma^*$.
\hfill $\lozenge$
\end{remark}

\begin{lemma}\label{L:Ratio-into-series}
Let $\gamma \neq 2$ and $|q| < \dfrac{\gamma}{|\gamma-2|}$, so that $J(\delta_{1-q}(\gamma),\gamma,k)$ is finite for $k \ge 0$.
\begin{enumerate}
\item If $\Phi(r,q) < 1$ for $r>q$, then $k \mapsto J(\delta_{1-q}(\gamma),\gamma,k)$ is strictly decreasing on the non-negative integers.
\item If $\Phi(r,q) > 1$ for $r>q$, then $k \mapsto J(\delta_{1-q}(\gamma),\gamma,k)$ is strictly increasing on the non-negative integers.
\end{enumerate}
\end{lemma}
\begin{proof}
Strict increasing (resp. decreasing) behavior of the function $k \mapsto J(\delta_{1-q}(\gamma),\gamma,k)$ would follow by showing the below ratio is strictly greater than 1 (resp. strictly less than 1) for non-negative integers $k$.
Throughout the proof, set $a=1-q$:
\begin{align}\label{E:AppendixSymbolRatio1}
\frac{J(\delta_a(\gamma),\gamma,k+1)}{J(\delta_a(\gamma),\gamma,k)} 
&= \frac{\Gammaf{\frac{2k+4+(\gamma-2)a}{\gamma}} \Gammaf{2k+4 -\frac{2k+4+(\gamma-2)a}{\gamma}} }{\Gammaf{\frac{2k+2+(\gamma-2)a}{\gamma}} \Gammaf{2k+2 -\frac{2k+2+(\gamma-2)a}{\gamma}}} \left( \frac{\gamma}{2k+2}\right)^2 (\gamma - 1)^{\frac{2}{\gamma}-2} \notag
\end{align}
 
By the Hölder symmetry of the symbol function (Theorem \ref{T:Holder-symmetry-symbol}), it is sufficient to restrict analysis to $\gamma \in (1,2)$; upon making the substitution $x = \frac{2}{\gamma}$, this is equivalent to analyzing the behavior of the following function for $x \in (1,2)$:
\begin{equation*}
\frac{\Gammaf{(k+2-a)x+a} \Gammaf{(k+2-a)(2-x)+a}}{\Gammaf{(k+1-a)x+a} \Gammaf{(k+1-a)(2-x)+a} } \cdot \frac{x^{-x} (2-x)^{x-2}}{(k+1)^2} := e^{A(k,x)}
\end{equation*}
Upon taking a logarithm, observe that $A(k,1) \equiv 0$, and further that $A(k,x)$ can be written
\begin{align*}
A(k,x) = B(k,x) + B(k,2-x),
\end{align*}
where
\begin{equation}\label{E:DefOfAppendixFcnB}
B(k,x) := \log\left[\frac{\Gammaf{(k+2-a)x+a}}{\Gammaf{(k+1-a)x+a}} \right] - x\log x - \log(k+1).
\end{equation}

For the remainder of the proof we focus on case (2), noting that the same argument with trivial modifications in appropriate places will prove case (1).

We have just seen that the statement
\begin{equation}\label{E:imporant-ratio}
\frac{J(\delta_a(\gamma),\gamma,k+1)}{J(\delta_a(\gamma),\gamma,k)} > 1
\end{equation}
for all non-negative integers $k$ and $\gamma \in (1,2)$ is equivalent to the statement that $A(k,x) > 0$ for all non-negative integers $k$ and $x \in (1,2)$.

Since $A(k,1) \equiv 0$ for all non-negative integers $k$, the desired positivity of $A(k,x)$ will follow by establishing a stronger condition, namely, that for $x \in (1,2)$
\begin{equation}\label{E:dA/dx>0}
\frac{\partial A}{\partial x}(k,x) = \frac{\partial B}{\partial x}(k,x) - \frac{\partial B}{\partial x}(k,2-x) > 0.
\end{equation}

Estimate \eqref{E:dA/dx>0} will in turn follow if it can be established that $\frac{\partial B}{\partial x}(k,x)$ is strictly increasing for $x \in (1,2)$, since $0 < 2-x <x$ for such $x$.
It would thus suffice to show
\begin{equation}\label{E:AppendixB''>0}
\frac{\partial^2 B}{\partial x^2}(k,x) >0.
\end{equation}

Recalling the representation \eqref{E:polygamma-as-series} of $\psi'$ by a series, calculation now shows
\begin{align}
\frac{\partial^2 B}{\partial x^2}(k,x) 
&= (k+2-a)^2 \psi'((k+2-a)x+a) - (k+1-a)^2 \psi'((k+1-a)x+a) - \frac{1}{x} \notag\\
&= \sum_{j=1}^\infty \frac{(k+2-a)^2}{(j+(k+2-a)x + (a-1))^2} - \sum_{j=1}^\infty \frac{(k+1-a)^2}{(j+(k+1-a)x + (a-1))^2} - \frac{1}{x} \notag\\
&= D(k+1,x) - D(k,x), \label{E:AppendixB''inTermsOfD}
\end{align}
where
\begin{equation}\label{E:DefOfAppendixFcnD}
D(k,x) := \sum_{j=1}^\infty \frac{(k+1-a)^2}{(j+(k+1-a)x + (a-1))^2} - \frac{k}{x}.
\end{equation}

Temporarily regard $k$ as a continuous variable.
By \eqref{E:AppendixB''inTermsOfD}, it would now be sufficient to show that $D(k,x)$ increases as a function of $k \in [0,\infty)$ in order to conclude that estimate \eqref{E:AppendixB''>0} holds.  
We claim that for $x \in (1,2)$,
\begin{align}
\frac{\partial D}{\partial k}(k,x) 
&= \sum_{j=1}^\infty \frac{\partial}{\partial k}\left[ \frac{(k+1-a)^2}{(j+(k+1-a)x + (a-1))^2} \right] - \frac{1}{x} \notag \\
&= \sum_{j=1}^\infty \frac{2(k+1-a)(j+a-1)}{(j+(k+1-a)x+(a-1))^3}  - \frac{1}{x} \quad >0. \label{E:dD/dk>0}
\end{align}
The inequality in \eqref{E:dD/dk>0} is equivalent to saying that, for $k \in [0,\infty)$ and $x\in(1,2)$,
\begin{equation}\label{E:new-sum-ineq}
\sum_{j=1}^\infty \frac{2x(k+1-a)(j+a-1)}{(j+(k+1-a)x+(a-1))^3}  >1.
\end{equation}
Now substituting $q=1-a$ and $r = x(k+1-a) = x(k+q)$. 
The inequality \eqref{E:new-sum-ineq} would follow by showing that, for $r>q$,
\begin{equation}\label{E:desired-ineq}
\sum_{j=1}^\infty \frac{2r(j-q)}{(j+r-q)^3} = \Phi(r,q)  >1.
\end{equation}

This is the condition listed in case (2) of the theorem.
Retracing our steps, we see that inequality \eqref{E:desired-ineq} implies inequality \eqref{E:dD/dk>0}.
Now, by \eqref{E:AppendixB''inTermsOfD}, this implies inequality \eqref{E:AppendixB''>0}.
This implies \eqref{E:dA/dx>0}, which in turn shows \eqref{E:imporant-ratio}, meaning that $k \mapsto J(\delta_{1-q}(\gamma),\gamma,k)$ is an increasing function on the non-negative integers.

To prove case (1), return to \eqref{E:imporant-ratio} and retrace the same steps, changing ``$>$" to ``$<$" and ``increases" to ``decreases" in all necessary places.
\end{proof}

%%%%%%%%%%%%%%%%%%%%%%%%%%%%%%%%%%%%%%%%%%%%%%%%%%
%%%%%%%%%%%%%%%%%%%%%%%%%%%%%%%%%%%%%%%%%%%%%%%%%%

\section{Polygamma inequalities and complete monotonicity}\label{S:series-analysis}

\subsection{An initial estimate}\label{SS:an-initial-estimate}

We now examine the properties of
\begin{equation*}
\Phi(r,q) = 2r\psi'(r+1-q)+r^2\psi''(r+1-q).
\end{equation*}
We start from well-known upper and lower estimates on the polygamma functions; see \cite[Theorem 3]{GuoQiSri_2012}. 
If $x>0$ and $m$ is a positive integer, then
\begin{equation}\label{E:basic-polygamma-ineq}
\frac{(m-1)!}{x^m}+\frac{m!}{2x^{m+1}} < (-1)^{m+1}\psi^{(m)}(x) < \frac{(m-1)!}{x^m}+\frac{m!}{x^{m+1}}.
\end{equation}

We use this to describe the behavior of $\Phi(r,q)$ as $r\to\infty$. (Compare with Figure \ref{Im:plots-of-Phi(r,q)}.)
\begin{proposition}
Let q be a fixed real number. Then
\begin{equation*}
\lim_{r \to \infty} \Phi(r,q) = 1.
\end{equation*}
\end{proposition}
\begin{proof}
Directly from \eqref{E:basic-polygamma-ineq}, we have for $r > \max\{q-1,0\}$,
\begin{equation*}
\frac{2r^2-2rq+3r}{(r+1-q)^2} < 2r \psi'(r+1-q) < \frac{2r^2-2rq+4r}{(r+1-q)^2},
\end{equation*}
and 
\begin{equation*}
-\frac{(r-q+3)r^2}{(r+1-q)^3} < r^2 \psi''(r+1-q) < -\frac{(r-q+2)r^2}{(r+1-q)^3}.
\end{equation*}
Now combine these:
\begin{equation}\label{E:simple-estimate-on-Phi}
\frac{r^3+(2-3q)r^2+(3-5q+2q^2)r}{(r+1-q)^3}
< \Phi(r,q) < 
\frac{r^3+(4-3q)r^2+(4-6q+2q^2)r}{(r+1-q)^3}.
\end{equation}
Taking the limit as $r\to \infty$ gives the result.
\end{proof}

\begin{remark}
Given $q$, we must show either $\Phi(r,q)>1$ for $r>q$, or $\Phi(r,q)<1$ for $r>q$ in order to apply Lemma \ref{L:Ratio-into-series}.
This requires better estimates than those obtained in \eqref{E:simple-estimate-on-Phi}.
This is easily seen when $q=0$, in which case
\begin{equation*}
\frac{r^3+2r^2+3r}{r^3+3r^2+3r+1} < \Phi(r,0) < \frac{r^3+4r^2+4r}{r^3+3r^2+3r+1}.
\end{equation*}
This lower bound is always less than 1, while the upper bound is eventually greater than 1. 
Similarly, neither the upper nor lower bound in \eqref{E:simple-estimate-on-Phi} is strong enough to warrant application of Lemma \ref{L:Ratio-into-series} for any $q$.
Our goal is now to improve the estimates on $\Phi(r,q)$.
\hfill $\lozenge$
\end{remark}

\subsection{The Bernstein-Widder theorem}\label{SS:Berstein-Widder}

Let us recall another well-known formula for the polygamma functions $\psi^{(m)}$; see \cite[Equation 6.4.1]{AbrSteBook}:
\begin{equation}\label{E:Polygamma-integral-formula}
\psi^{(m)}(r) = (-1)^{m+1}\int_0^\infty \frac{t^m e^{-rt}}{1-e^{-t}}\,dt.
\end{equation}

Let $f$ be a real-valued function defined on $(c,\infty)$, for some $c \in \R$.
We say that $f$ is {\em strictly completely monotone} on $(c,\infty)$ if it is of class $C^\infty$ and
\begin{equation}\label{E:def-completely-monotone}
(-1)^{m} \frac{d^m}{dr^m}f(r) > 0
\end{equation}
for all non-negative integers $m$ and $r>c$.

Strictly completely monotone functions are characterized by the following theorem, which can be found in \cite[page 161]{WidderBook}.
\begin{theorem}[Bernstein-Widder]\label{T:Bernstein-Widder}
A function $f$ is strictly completely monotone on $(c,\infty)$ if and only if $f(r-c)$ is the Laplace transform of a finite positive Borel measure $\mu$ on $(c,\infty)$.
In other words,
\begin{equation*}
f(r-c) = \int_0^\infty e^{-rt}\,d\mu(t).
\end{equation*}

\end{theorem}

Our goal is to use Bernstein-Widder to analyze $\Phi(r,q)$ and closely related functions.

Let us define
\begin{subequations}
\begin{equation}\label{E:def-of-Theta}
\Theta(r,q) = r^2 \psi'(r+1-q)
\end{equation}
and observe that 
\begin{equation}\label{E:r-deriv-of-Theta-is-Phi}
\frac{\dee \Theta}{\dee r}(r,q) = \Phi(r,q).
\end{equation}
\end{subequations}

\begin{lemma}\label{L:lemma-on-form-of-Theta}
For $x>0$ we have the following
\begin{equation}\label{E:integral-expression-for-Theta}
\Theta(x+q,q) = x + 2q -\frac{1}{2} + \int_0^\infty \frac{M(t,q)}{(e^t-1)^3} e^{-xt}\, dt,
\end{equation}
where 
\begin{equation}\label{E:def-of-M(t,q)}
M(t,q) = g_0(t) - g_1(t)q + g_2(t)q^2,
\end{equation}
and $g_0,g_1,g_2$ are positive functions on $(0,\infty)$.
Explicitly,
\begin{subequations}\label{E:def-of-g_j}
\begin{align}
g_0(t) &= e^t(2-2e^t+t+t e^t),\\
g_1(t) &= 2(e^t-1)(1-e^t+t e^t),\\
g_2(t) &= t(e^t-1)^2.
\end{align}
\end{subequations}
\end{lemma}
\begin{proof}
From definition \eqref{E:def-of-Theta} and \eqref{E:Polygamma-integral-formula} we have
\begin{align*}
\Theta(x+q,q) &= (x+q)^2 \int_0^\infty \frac{t e^{-(x+1)t}}{1-e^{-t}}\,dt = \int_0^\infty \frac{t e^{qt}}{e^{t}-1} (x+q)^2 e^{-(x+q)t}\,dt. \label{E:Theta-expn-step1}
\end{align*}
Integrate by parts twice, setting 
\[
u_1 = \frac{t e^{qt}}{e^{t}-1},\qquad dv_1 = (x+q)^2 e^{-(x+q)t}\,dt,
\]
for the first application, and $u_2 = \frac{du_1}{dt}$ and $dv_2 = v_1\,dt$ for the second.
This means that
\begin{equation*}
u_2 = \frac{e^t-1-te^t-(t-te^t)q}{(e^t-1)^2}\, e^{qt}, \qquad v_1 = -(x+q) e^{-(x+q)t},
\end{equation*}
and consequently,
\begin{equation*}
du_2 = \frac{g_0(t) - g_1(t)q + g_2(t)q^2}{(e^t-1)^3}\, e^{qt} \, dt, \qquad v_2 = e^{-(x+q)t},
\end{equation*}
where $g_0,g_1,g_2$ are given in the statement of the theorem.
Now,
\begin{align}
\Theta(x+q,q) = \int_0^\infty u_1\,dv_1 &= u_1 v_1 \Big\vert_0^\infty -  \int_0^\infty v_1\,du_1 \notag \\
&= u_1 v_1 \Big\vert_0^\infty - u_2 v_2 \Big\vert_0^\infty + \int_0^\infty v_2\,du_2 \notag \\
&= (x+q) - \big(\tfrac{1}{2}-q\big) + \int_0^\infty \frac{M(t,q)}{(e^t-1)^3} \,e^{-xt}\, dt. \label{E:Theta-expn-step2}
\end{align}

Using the definitions of $M(t,q)$ and the $g_j(t)$ given above, we see that for $t$ near 0,
\[
M(t,q) = \Big(q^2 - q+\frac{1}{6}\Big)t^3 + \Big(q^2 - \frac{7q}{6}+\frac{1}{4}\Big)t^4 + O\big(t^5\big)
\]
On the other hand, if $t$ is large enough (say, $t>1$), there is a constant $C_q$ such that
\[
|M(t,q)| \le C_q\, t e^{2t}.
\]
From these estimates on $M(t,q)$ we see that the integral appearing in \eqref{E:Theta-expn-step2} converges for $x>0$.
It now only remains to show the positivity of the $g_j(t)$ for $t>0$.

Let us first set $h_0(t) = 2-2e^t+t+t e^t$, so that $g_0(t) = e^t h_0(t)$.
Calculating derivatives,
\begin{equation*}
h_0'(t) = 1 + (t-1)e^t, \qquad h_0''(t) = te^t,
\end{equation*}
and so $0 = h_0(0) = h'_0(0) = h''(0)$. 
Also, clearly $h_0''(t)>0$, for $t>0$ which now implies that $h_0'(t)>0$, which in turn implies $h_0(t)>0$.
Thus $g_0(t)>0$ for $t>0$.

Now let $h_1(t) = 1-e^t+te^t$, so that $g_1(t) = 2(e^t-1)h_1(t)$.
We see that
\[
h_1'(t) = t e^t,
\]
and so $0 = h_1(0) = h_1'(0)$.
Clearly $h_1'(t)>0$ for $t>0$, which implies $h_1(t)>0$ and therefore $g_1(t)>0$ for $t>0$.

Finally, it is immediate from the formula that $g_2(t)>0$ for $t>0$.
\end{proof}

\subsection{Completely monotone functions}\label{SS:completely-monotone-functions}
For real $q$ define the following function
\begin{equation}\label{E:def-of-Fq}
F_q(x) = \Theta(x+q,q)-x-2q+\frac{1}{2}.
\end{equation}

\begin{theorem}\label{T:Theta-et-al-is-completely-monotone-q>=1}
Let $q \in (-\infty,0] \cup [1,\infty)$. 
The function $x \mapsto F_q(x)$ is strictly completely monotone on $x>0$.
\end{theorem}
\begin{proof}
From Lemma \ref{L:lemma-on-form-of-Theta}, we have
\begin{equation}\label{E:Fq-integral-representation}
F_q(x) = \int_0^\infty \frac{M(t,q)}{(e^t-1)^3}\, e^{-xt}\, dt,
\end{equation}
where $M(t,q) = g_0(t) - g_1(t)q + g_2(t)q^2$, and the $g_j$ are given in \eqref{E:def-of-g_j}.
%For a fixed $t>0$, we know that $g_2(t) = t(e^t-1)^2>0$, and so $q \mapsto M(t,q)$ is positive for all but at most a finite interval of $q$ values.

Regarding $M(t,q)$ as a quadratic in the $q$ variable, the discriminant is
\begin{align}
\Delta(t) &= g_1(t)^2 - 4 g_0(t)g_2(t) \notag\\
&= 4 - 16 e^t + 24 e^{2t} - 16 e^{3t} + 
4 e^{4t} + (-4 e^t + 8 e^{2t} - 4 e^{3t}) t^2 \notag \\
&= 4 (e^t-1)^2 (1 - 2 e^t + e^{2t} - e^t t^2). \label{E:quadratic-discriminant}
\end{align}
The roots to $M(t,q) = 0$ are therefore given by taking
\[
q = \frac{g_1(t) \pm \sqrt{\Delta(t)}}{2g_2(t)};
\]
we now label these roots as
\begin{subequations}
\begin{align}
s_1(t) &= \frac{t e^t + 1 - e^t + \sqrt{1 - 2 e^t + e^{2t} - t^2 e^t }}{t(e^t-1)} \label{E:root1-s1},\\
s_2(t) &= \frac{t e^t + 1 - e^t - \sqrt{1 - 2 e^t + e^{2t} - t^2 e^t }}{t(e^t-1)}\label{E:root2-s2}.
\end{align}
\end{subequations}
Observe that the term under the square root is positive for $t>0$:
\begin{align*}
1 - 2 e^t + e^{2t} - t^2 e^t > 0 \qquad \Longleftrightarrow \qquad \cosh{t} > 1 + \frac{t^2}{2} \qquad \Longleftrightarrow \qquad \sum_{j=2}^\infty \frac{t^{2j}}{(2j)!} > 0.
\end{align*}

We now claim for $t>0$ that 
\begin{equation}\label{E:-1<s2<s1<0}
0 < s_2(t) < s_1(t) < 1.
\end{equation}
The fact that $s_2(t) < s_1(t)$ is immediate. 
The inequality $s_2(t) > 0$ follows from the fact that $M(t,q) > 0$ for all $q \le 0$ and $t>0$ (since each $g_j(t) > 0$ for $t>0$).

For the remaining inequality we show that it is both true and sharp.
Indeed,
\begin{align*}
s_1(t) < 1 \qquad &\Longleftrightarrow \qquad te^t - t > t e^t + 1 - e^t + \sqrt{1 - 2 e^t + e^{2t} - t^2 e^t} \notag \\
&\Longleftrightarrow \qquad e^t - 1 - t > \sqrt{1 - 2 e^t + e^{2t} - t^2 e^t} \notag \\
&\Longleftrightarrow \qquad  (e^t-1)^2 - 2 t (e^t-1) + t^2 > (e^t-1)^2 - t^2 e^t \notag \\
&\Longleftrightarrow \qquad (t-2)e^t + t + 2 > 0. \label{E:def-of-polynomial-E}
\end{align*}
Now set $E(t) := (t-2)e^t + t + 2$.
Then 
\[
E'(t) = (t-1)e^t + 1, \qquad E''(t) = t e^t,
\]
and so $0 = E(0) = E'(0) = E''(0)$. 
Since $E''(t)>0$ for $t>0$, we have that $E'(t)>0$ which in turn implies that $E(t)>0$, thus confirming that $s_1(t) < 1$.
(Notice that the function $E$, along with the same line of reasoning given here, appeared in Lemma \ref{L:lemma-on-form-of-Theta}, where the function was called $h_0$.)

For the sharpness of this inequality, observe that
\begin{align*}
\lim_{t \to \infty} s_1(t) = \lim_{t \to \infty} \frac{t e^t}{t e^t} \cdot \frac{1 + t^{-1}(e^{-t} - 1) +\sqrt{t^{-2} (1-e^{-t})^2 - e^{-t}})}{1-e^{-t}} = 1.
\end{align*}

Now since $0 < s_2(t) < s_1(t) < 1$ for all $t>0$, we conclude that $M(t,q) > 0$ for $t>0$ and $q \in (-\infty,0]\cup[1,\infty)$.
The Bernstein-Widder theorem now shows that $F_q(x)$ is strictly completely monotone for $q$ in this range.
\end{proof}

\begin{remark}\label{R:improving-the-endpoint}
The proof of Theorem \ref{T:Theta-et-al-is-completely-monotone-q>=1} shows the right endpoint $q=1$ of the interval $0 < s_2(t) < s_1(t) < 1$ is sharp.
No claim of sharpness is made for the left endpoint, however.
Numerical evidence suggests that $s_2(t)$ is a strictly increasing function on $t>0$, which would imply that the left endpoint can be slightly improved from $q=0$ to
\begin{equation*}
q = \lim_{t \searrow 0} s_2(t) = \frac{3-\sqrt{3}}{6} \approx 0.211325.
\end{equation*}
In any case, we can say that $F_q(x)$ fails to be completely monotone for $q \in \big(\tfrac{3-\sqrt{3}}{6},1\big)$.
\hfill $\lozenge$
\end{remark}

\begin{corollary}\label{C:Phi-estimate-q>=0}
Choose $q \in (-\infty,0]\cup[1,\infty)$. 
Then $\Phi(r,q) < 1$ for $r>q$.
\end{corollary}
\begin{proof}
Fix $q \in (-\infty,0]\cup[1,\infty)$ and set $x=r-q$. 
Theorem \ref{T:Theta-et-al-is-completely-monotone-q>=1} now says the function
\[
r \mapsto F_q(r-q) =  \Theta(r,q) - r - q + \frac{1}{2}
\]
is strictly completely monotone for $r>q$. 
Now differentiate in $r$ to see that
\[
\frac{\dee \Theta}{\dee r}(r,q) - 1 < 0,
\]
for all $r>q$.
Since $\frac{\dee \Theta}{\dee r}(r,q) = \Phi(r,q)$ by \eqref{E:r-deriv-of-Theta-is-Phi}, we obtain the result.
\end{proof}

\begin{remark}
Notice that $q=\frac{2}{3}$ (the $q$-value corresponding to the preferred symbol function) lies in the interval for which complete monotonicity of $F_q(x)$ is known to fail; see Remark \ref{R:improving-the-endpoint}.
This means the Berstein-Widder approach is not applicable, and the preferred symbol function must be handled using other means (see Section \ref{S:EM-pref-symbol-function}).
\hfill $\lozenge$
\end{remark}

\subsection{Consequences for the Leray transform}\label{SS:Consequences-for-Leray}

The results in Sections \ref{SS:Berstein-Widder} and \ref{SS:completely-monotone-functions} are now combined with Lemma \ref{L:Ratio-into-series}.
Recall (Corollary \ref{C:Leray-boundedness}) that the Leray transform is bounded on $L^2(M_\gamma,\mu_d)$ if and only if $d \in (-1,2\gamma-1) = \ci_0(\gamma)$.
We now see that the norms of the sub-Leray operators $\bm{L}_k$ are {\em strictly decreasing} in $k$ for a range of $d$ values with a combined length of more than half the length of $\ci_0(\gamma)$.

\begin{theorem}\label{T:Leray-boundedness-from-Bernstein-Widder}
Let $\bm{L}_k$ denote the sub-Leray operator for each non-negative integer $k$.
\begin{enumerate}
\item If $\gamma>2$ and $d \in (-1,1] \cup [\gamma-1,2\gamma-1)$, then the function $k \mapsto J(d,\gamma,k)$ is strictly decreasing on the non-negative integers.
Thus
\[
\norm{\bm{L}_k}_{L^2(M_\gamma,\mu_{d})} > \norm{\bm{L}_{k+1}}_{L^2(M_\gamma,\mu_{d})}.
\]
\item If $\gamma<2$ and $d \in (-1,\gamma-1] \cup [1,2\gamma-1)$, then the function $k \mapsto J(d,\gamma,k)$ is strictly decreasing on the non-negative integers.
Thus
\[
\norm{\bm{L}_k}_{L^2(M_\gamma,\mu_{d})} > \norm{\bm{L}_{k+1}}_{L^2(M_\gamma,\mu_{d})}.
\]
\end{enumerate}
Consequently, in both settings, the norm of the full Leray transform is 
\[
\norm{\bm{L}}_{L^2(M_\gamma,\mu_d)} = \norm{\bm{L}_0}_{L^2(M_\gamma,\mu_d)}.
\]
\end{theorem}

\begin{proof}
By Corollary \ref{C:Phi-estimate-q>=0}, $\Phi(r,q)<1$ for $q \in (-\infty,0] \cup [1,\infty)$ and $r>q$.
On the other hand, Lemma \ref{L:Control-on-q-for-boundedness} says that the symbol function $J(\delta_{1-q}(\gamma),\gamma,k)$ is finite for all non-negative integers $k$ if and only if $|q|< \frac{\gamma}{|\gamma-2|}$. 

Upon intersecting these two intervals, Lemma \ref{L:Ratio-into-series} says that for $\gamma \neq 2$, the function $k \mapsto J(\delta_{1-q}(\gamma),\gamma,k)$ is strictly decreasing on the non-negative integers for 
\begin{equation*}\label{E:range-of-q-strict-decreasing-subLeray}
q \in \Big( \tfrac{-\gamma}{|\gamma-2|},0 \Big] \cup \Big[1, \tfrac{\gamma}{|\gamma-2|} \Big).
\end{equation*}

Let us write
\begin{equation}\label{E:d-in-terms-of-q-strict-decreasing-subLeray}
d = \delta_{1-q}(\gamma) = (1-q)(\gamma-2)+1
\end{equation}
and recall that $\sqrt{J(d,\gamma,k)} = \norm{\bm{L}_k}_{L^2(M_\gamma,\mu_{d})}$ by equation \eqref{E:norm_Lk-sqrt}.

Now consider four separate cases:

When $\gamma>2$ and $q \in \big[1, \tfrac{\gamma}{\gamma-2} \big)$, equation \eqref{E:d-in-terms-of-q-strict-decreasing-subLeray} implies that $k \mapsto \norm{\bm{L}_k}_{L^2(M_\gamma,\mu_{d})}$ is strictly decreasing for $d \in (-1,1]$.

When $\gamma>2$ and $q \in \big( \tfrac{-\gamma}{\gamma-2},0 \big]$, equation \eqref{E:d-in-terms-of-q-strict-decreasing-subLeray} implies that $k \mapsto \norm{\bm{L}_k}_{L^2(M_\gamma,\mu_{d})}$ is strictly decreasing for $d \in [\gamma-1,2\gamma-1)$.

When $\gamma<2$ and $q \in \big[1, \tfrac{\gamma}{2-\gamma} \big)$, equation \eqref{E:d-in-terms-of-q-strict-decreasing-subLeray} implies that $k \mapsto \norm{\bm{L}_k}_{L^2(M_\gamma,\mu_{d})}$ is strictly decreasing for $d \in [1,2\gamma-1)$.

When $\gamma<2$ and $q \in \big( \tfrac{-\gamma}{2-\gamma}, 0\big]$, equation \eqref{E:d-in-terms-of-q-strict-decreasing-subLeray} implies that $k \mapsto \norm{\bm{L}_k}_{L^2(M_\gamma,\mu_{d})}$ is strictly decreasing for $d \in (-1,\gamma-1]$.

This establishes both (1) and (2). 
In all settings encompassed by these two cases, the $L^2(M_\gamma,d)$-norm of $\bm{L}_k$ decreases with $k$, implying that $\bm{L}_0$ is the sub-Leray operator with the biggest norm. 
Equation \eqref{E:Leray-norm-as-sup} now says that $\norm{\bm{L}}_{L^2(M_\gamma,d)} = \norm{\bm{L}_0}_{L^2(M_\gamma,d)}$.
\end{proof}

Both the pairing measure $\sigma = r^{\gamma-1}\,dr\w d\theta \w ds$ ($d=\gamma-1$) and Lebesgue measure $\mu_1 = r \,dr\w d\theta \w ds$ ($d=1$) fall within the range of applicability of Theorem \ref{T:Leray-boundedness-from-Bernstein-Widder}.
We now record the norms of the Leray transform in both settings:

\begin{corollary}\label{C:Leray-norm-pairing}
Let $\sigma = r^{\gamma-1}\,dr\w d\theta \w ds$.
The Leray transform is bounded on $L^2(M_\gamma,\sigma)$ with norm
\begin{equation*}
\norm{\bm{L}}_{L^2(M_\gamma,\sigma)} = \frac{\gamma}{2\sqrt{\gamma-1}}.
\end{equation*}
\end{corollary}
\begin{proof}
Since $\sigma$ corresponds to $d = \gamma-1$, Theorem \ref{T:Leray-boundedness-from-Bernstein-Widder} applies. Thus by \eqref{E:def-symbol-function},
\[
\norm{\bm{L}}_{L^2(M_\gamma,\sigma)} = \norm{\bm{L}_0}_{L^2(M_\gamma,\sigma)} = \sqrt{C_\sigma(\gamma,0)}  = \frac{\gamma}{2\sqrt{\gamma-1}}.
\]
Recall that the pairing symbol function $C_\sigma(\gamma,k) = J(\gamma-1,\gamma,k)$.
\end{proof}

\begin{remark}
Barrett and Edholm calculated $\norm{\bm{L}}_{L^2(M_\gamma,\sigma)}$ in \cite[Proposition 4.18]{BarEdh22} using a different approach.
The argument given there is tailored to the pairing measure ($q=0$) and seems not to easily generalize to other measure settings.
\hfill $\lozenge$
\end{remark}

\begin{corollary}\label{C:Leray-norm-Lebesgue}
Let $\mu_1 = r\, dr \w d\theta \w ds$.
The Leray transform $\bm{L}$ is bounded on $L^2(M_\gamma,\mu_1)$ with norm
\begin{equation*}
\norm{\bm{L}}_{L^2(M_\gamma,\mu_1)} = 
\begin{cases}
(\gamma-1)^{\frac{1}{\gamma}-1} \sqrt{ \frac{\pi}{4}(\gamma-2)\gamma \csc\big( \frac{2\pi}{\gamma} \big)}, & \gamma \in (1,2)\cup(2,\infty) \\
1, & \gamma = 2.
\end{cases}
\end{equation*}
The formula is continuous at $\gamma = 2$.
\end{corollary}
\begin{proof} 
Since $\mu_1$ corresponds to $d=1$, Theorem \ref{T:Leray-boundedness-from-Bernstein-Widder} applies once again for $\gamma \neq 2$:
\begin{align*}
\norm{\bm{L}}_{L^2(M_\gamma,\mu_1)} = \norm{\bm{L}_0}_{L^2(M_\gamma,\mu_1)} = \sqrt{J(1,\gamma,0)} &= (\gamma-1)^{\frac{1}{\gamma}-1} \sqrt{ \big( \tfrac{\gamma}{2} \big)^2 \, \Gamma\big( \tfrac{2}{\gamma} \big) \Gamma\big( 2- \tfrac{2}{\gamma} \big) } \\
&= (\gamma-1)^{\frac{1}{\gamma}-1} \sqrt{ \tfrac{\pi}{4} (\gamma-2)\gamma \csc\big( \tfrac{2\pi}{\gamma} \big) },
\end{align*}
where we've used the factorial property of the $\Gamma$-function as well as Euler's reflection formula.

The $\gamma = 2$ case follows from \eqref{E:symbol-function_gamma=2} and
L'Hôpital's rule confirms continuity.
\end{proof}

%%%%%%%%%%%%%%%%%%%%%%%%%%%%%%%%%%%%%%%%%%%%%%%%%%%%%%%%%%%%%%%%%%%%%%%%%%%%%%%%%%%%%
%%%%%%%%%%%%%%%%%%%%%%%%%%%%%%%%%%%%%%%%%%%%%%%%%%%%%%%%%%%%%%%%%%%%%%%%%%%%%%%%%%%%%
%%%%%%%%%%%%%%%%%%%%%%%%%%%%%%%%%%%%%%%%%%%%%%%%%%%%%%%%%%%%%%%%%%%%%%%%%%%%%%%%%%%%%

\section{The preferred symbol function}\label{S:EM-pref-symbol-function}

In this section we prove that the preferred symbol function $k \mapsto C_\nu(\gamma,k)$, $\gamma \neq 2$, is strictly {\em increasing}, in stark contrast with the strictly decreasing symbol function behavior seen in the earlier parts of the paper. 
Our goal is to show
\begin{equation*}
\Phi(r,\tfrac{2}{3}) > 1, \qquad \mathrm{for} \qquad r > \tfrac{2}{3}.
\end{equation*}
Unlike the work in Section \ref{S:series-analysis}, we are unable to invoke complete monotonicity and the Berstein-Widder theorem to prove this inequality, since it is known that the related function $F_q$ is not completely monotone for $q = \frac{2}{3}$; see Remark \ref{R:improving-the-endpoint}.

In place of the integral representations used in the previous section, we use here the infinite series description obtained from \eqref{E:Important-series}:
\begin{equation}\label{E:Phi_when_q=2/3}
\Phi\big(r,\tfrac{2}{3}\big) = \sum_{j=1}^\infty f_r(j), \qquad \mathrm{where} \qquad f_r(j) = \frac{18r(3j-2)}{(3r+3j-2)^3}.
\end{equation}

\subsection{Two tools}\label{SS:Two tools}

Two classical results are crucial to the subsequent analysis.

The first is {\em Descartes' Rule of Signs} (see \cite{Wang2004} for a simple proof):
\begin{proposition}[Descartes]\label{P:Descartes}
Let $p$ be a single variable polynomial with real coefficients, with monomial terms arranged so that exponents appear in ascending order. 
The number of positive roots of $p$ (counted with multiplicities) is either (i) equal to the number of sign changes between consecutive (non-zero) coefficients, or (ii) less than that by an even number.
\end{proposition}

The second is the {\em Euler-Maclaurin formula} relating sums and integrals; see \cite[Section 9.5]{GKPBook}.
For our purposes it is sufficient to use the following first-order version:

\begin{proposition}[Euler-Maclaurin]\label{P:Euler-Macl}
Let $m<n$ be integers and $f \in C^1[m,n]$. 
Then
\begin{equation}\label{E:EulerMacL1stOrder}
\sum_{j=m+1}^n f(j) 
= \int_m^n f(x)\,dx + \frac{f(n)-f(m)}{2} + \int_m^n f'(x) P_1(x)\,dx,
\end{equation}
where $P_1(x) = B_1(x-\lfloor x \rfloor)$, $B_1(x) = x-\tfrac{1}{2}$ is the first Bernoulli polynomial and $\lfloor x \rfloor$ the greatest integer less than or equal to $x$.
\end{proposition}

We refer to the rightmost integral in \eqref{E:EulerMacL1stOrder} as the {\em Bernoulli integral}.

\subsection{Applying Euler-Maclaurin}\label{SS:Application-of-EulerMaclaurin}

We use the Euler-Maclaurin formula to re-express the sum in \eqref{E:Phi_when_q=2/3}, then use carefully chosen bounds to estimate the Bernoulli integral. 
After a considerable amount of analysis, the behavior of $r \mapsto \Phi\big(r,\frac{2}{3}\big)$ is reduced to a problem about the positive real roots of an explicit polynomial of degree 16. 
This in turn is completely understood using Descartes' Rule of Signs and direct evaluation.

The computations in Section \ref{S:EM-pref-symbol-function} and the Appendix involve lengthy manipulations of polynomials and rational functions. While the reader could (in principle) work these out by hand, we highly recommend the use of software when verifying these computations.
We have made a detailed and heavily notated Mathematica notebook available at the Github page of the first author.
Please follow the link given in \cite{Edholm_Leray_nb}.

\begin{theorem}\label{T:EdhShe-series}
If $r>\tfrac{2}{3}$, then $\Phi\big(r,\tfrac{2}{3}\big) = 2r \psi'\big(r+\tfrac{1}{3}\big) + r^2 \psi''\big(r+\tfrac{1}{3}\big) > 1$.
\end{theorem}

\begin{proof}
Start with the function $f_r$ from \eqref{E:Phi_when_q=2/3} and use the first order Euler-Maclaurin formula \eqref{E:EulerMacL1stOrder} with $m=0$ and $n \to \infty$. 
Computation yields the first two terms on the right side of \eqref{E:EulerMacL1stOrder}:
\begin{equation}\label{E:EM-eval-A}
\int_0^\infty f_r(x)\,dx = 1-\frac{4}{(3r-2)^2}, \qquad 
\frac{f_r(\infty)-f_r(0)}{2} = \frac{18r}{(3r-2)^3},
\end{equation}
where $f_r(\infty) = \lim_{n\to\infty}f_r(n)$.
Computation also shows that
\begin{equation*}
f'_r(x) = \frac{54r(4+3r-6x)}{(3x+3r-2)^4},
\end{equation*}
from which the following expression for the Bernoulli integral in \eqref{E:EulerMacL1stOrder} is calculated:
\begin{equation*}
\int_0^\infty f'_r(x) P_1(x)\,dx = \sum_{N=0}^{\infty} \int_0^1 f_r'(x+N) B_1(x)\,dx = \sum_{N=0}^\infty \frac{81r(9N^2-3N-9r^2-2)}{(3r+3N+1)^3(3r+3N-2)^3}.
\end{equation*}
Denote the rational function appearing in the sum above by
\begin{equation}\label{E:DefOfS_r}
S(r,N) = \frac{81r(9N^2-3N-9r^2-2)}{(3r+3N+1)^3(3r+3N-2)^3}.
\end{equation}
Now inserting \eqref{E:EM-eval-A} and \eqref{E:DefOfS_r} into \eqref{E:EulerMacL1stOrder}, we see by \eqref{E:Phi_when_q=2/3} that
\begin{align}
\Phi\big(r,\tfrac{2}{3} \big) = \sum_{j=1}^\infty f_r(j) 
&= 1-\frac{4}{(3r-2)^2} + \frac{18r}{(3r-2)^3} + \sum_{N=0}^\infty S(r,N) \notag \\
&= 1 + \frac{6r-1}{(3r+1)^3} + \sum_{N=1}^\infty S(r,N). \label{E:AppQuant0}
\end{align}  
(Note that we have peeled off the $N=0$ term from the summation.)

It thus remains to show that
\begin{equation}\label{E:AppQuant1}
\frac{6r-1}{(3r+1)^3} + \sum_{N=1}^\infty S(r,N)
\end{equation}
is strictly positive for $r>\frac{2}{3}$.

\subsubsection{The function $S(r,x)$}\label{SSS:Analysis-on-S}

We now show that $x \mapsto S(r,x)$ ($x$ is for now regarded as a non-negative real variable) has a single local extrema (a maximum) on $[0,\infty)$.  
Indeed,
\begin{align*}
\frac{\dee S}{\dee x}(r,x) = \frac{243r[(-4 + 39 r - 36 r^2 + 
 162 r^3) + (18 + 18 r + 216 r^2) x + (54 - 54 r) x^2 - 108 x^3]}{(3r+3x+1)^4 (3r+3x-2)^4}
\end{align*}

Write the cubic polynomial in $x$ appearing in the brackets above by
\begin{equation}\label{E:def-of-p_r}
p_r(x) := (-4 + 39 r - 36 r^2 + 
 162 r^3) + (18 + 18 r + 216 r^2) x + (54 - 54 r) x^2 - 108 x^3.
\end{equation}

It is easily checked that for $r > \frac{2}{3}$ the coefficients of $1$ and $x$ are positive, while the coefficient of $x^2$ changes sign at $r=1$, and the coefficient of $x^3$ is a negative constant.
Thus for all $r$ in this range, there is exactly one sign change in consecutive coefficients when the monomial terms of $p_r(x)$ are considered in ascending order. 
Descartes' Rule of Signs thus guarantees a {\em unique} positive real root of the function $x \mapsto p_r(x)$, where the values of $p_r(x)$ change from positive to negative.
Denote this root by $Q_r$. 

We claim that $Q_r \in \big(\frac{3r}{2}+\frac{1}{6},\frac{3r}{2}+\frac{1}{3}\big)$: Indeed, for $r>\frac{2}{3}$,
\begin{align*}
&p_r\big(\tfrac{3r}{2}+\tfrac{1}{6}\big) = 81r  > 0, \\
&p_r\big(\tfrac{3r}{2}+\tfrac{1}{3}\big) = 4+66r-\tfrac{225}{2}r^2 < 0. 
\end{align*}
(A more precise estimate on $Q_r$ is obtained in Lemma \ref{L:Qr-bound-appendix}.)

We see from \eqref{E:DefOfS_r} that $x \mapsto S(r,x)$ starts negative, then becomes (and remains) positive for $x$ large enough, since the function tends to zero as $x \to \infty$ and its derivative changes signs exactly once. 
In particular, the global maximum value of $S(r,Q_r)$ is positive.

We now trivially re-write the quantity in \eqref{E:AppQuant1} as
\begin{equation}\label{E:AppQuant2}
\frac{6r-1}{(3r+1)^3} + S(r,1) + S(r,2) + \sum_{N=3}^\infty S(r,N).
\end{equation}
(Separating out the $N=1$ and $N=2$ terms from the rest of the summation simplifies the following estimate.)
We claim that
\begin{equation}\label{E:Lower-estimate1}
\sum_{N=3}^\infty S(r,N) > \int_2^\infty S(r,x)\,dx - S(r,Q_r).
\end{equation}
To see this, set $K := \lfloor Q_r \rfloor$, the greatest integer less than or equal to $Q_r$. 
From the discussion above, $S(r,\cdot)$ increases on $(1,K)$ and decreases on $(K+1,\infty)$.  
We now consider two cases: $K\le2$ and $K\ge 3$.

When $K \le 2$, we have that $S(r,\cdot)$ decreases on $(3,\infty)$, so
\begin{subequations}
\begin{equation}\label{E:Lower-estimate2}
\sum_{N=3}^\infty S(r,N) > \int_3^\infty S(r,x)\,dx.
\end{equation}
But also note that
\begin{equation}\label{E:Lower-estimate3}
0 > \int_2^3 S(r,x)\,dx - S(r,Q_r),
\end{equation}
\end{subequations}
since any integral over an interval of unit length is overestimated by the maximum value of its integrand.
Combining \eqref{E:Lower-estimate2} with \eqref{E:Lower-estimate3} now yields \eqref{E:Lower-estimate1} for $K \le 2$.

When $K \ge 3$, we are able to write
\begin{subequations}\label{E:Sum-IntegralEstimates}
\begin{align}
\sum_{N=3}^{K} S(r,N) &> \int_2^K S(r,x)\,dx, \label{E:Sum-IntegralEstimates-a}\\
\sum_{N=K+1}^\infty S(r,N) &> \int_{K+1}^\infty S(r,x)\,dx. \label{E:Sum-IntegralEstimates-b}
\end{align}
Using reasoning identical to what justified \eqref{E:Lower-estimate3}, we see that
\begin{equation}\label{E:Sum-IntegralEstimates-c}
0 > \int_K^{K+1} S(r,x)\,dx - S(r,Q_r).
\end{equation}
\end{subequations}
Combining \eqref{E:Sum-IntegralEstimates-a}, \eqref{E:Sum-IntegralEstimates-b} and \eqref{E:Sum-IntegralEstimates-c} yields \eqref{E:Lower-estimate1} for $K \ge 3$.

From \eqref{E:Lower-estimate1}, the following quantity is a lower bound on \eqref{E:AppQuant1} = \eqref{E:AppQuant2}:
\begin{equation}\label{E:New-lower-estimate-a}
\frac{6r-1}{(3r+1)^3} + S(r,1) + S(r,2) + \int_2^\infty S(r,x)\,dx - S(r,Q_r).
\end{equation}
We now show \eqref{E:New-lower-estimate-a} is positive by computing the above integral and estimating $S(r,Q_r)$.

A partial fraction decomposition of $S(r,x)$ helps to yield the following:
\begin{equation}\label{E:integral-comp-logs}
\int_2^\infty S(r,x)\,dx = \frac{3r(108r^3+594r^2+1035r+616)}{2(3r+4)^2(3r+7)^2} - 2r\log\bigg(\frac{3r+7}{3r+4} \bigg).
\end{equation}

We also have the following estimate on $S(r,Q_r)$: for $r>\frac{2}{3}$,
\begin{equation}\label{E:estimate-on-S(r,Qr)}
S(r,Q_r) < \frac{16}{3125 r^3}.
\end{equation}
The proof this estimate is given in Lemma \ref{L:estmate-S(r,Qr)} of the Appendix.
(The right hand side of \eqref{E:estimate-on-S(r,Qr)} is the leading term in the Taylor expansion of $S(r,Q_r)$ at $\infty$; see Remark \ref{R:expansion-of-Q-at-infty}.)

We now use \eqref{E:integral-comp-logs} and \eqref{E:estimate-on-S(r,Qr)} to define a new function:
\begin{align}
H(r) &:=  \frac{6r-1}{(3r+1)^3} + S(r,1) + S(r,2) + \int_2^\infty S(r,x)\,dx - \frac{16}{3125 r^3} \notag \\
&= \frac{6r-1}{(3r+1)^3} + \frac{81r(4-9r^2)}{(3r+1)^3(3r+4)^3} + \frac{81r(28-9r^2)}{(3r+4)^3(3r+7)^3} + \label{E:DefOfH(r)} \\
& \qquad \qquad + \frac{3r(108r^3+594r^2+1035r+616)}{2(3r+4)^2(3r+7)^2} - 2r\log\bigg(\frac{3r+7}{3r+4}\bigg) - \frac{16}{3125 r^3} \notag
\end{align}

Upon combining \eqref{E:integral-comp-logs} with the bound \eqref{E:estimate-on-S(r,Qr)}, we see that $H(r)$ is a lower bound for \eqref{E:New-lower-estimate-a}.
Our goal is now to show $H(r)$ is positive for $r>\frac{2}{3}$.
Note that while $H$ is itself not a rational function, its second derivative is rational.
This is crucial to the coming argument.

\subsubsection{The function $H(r)$}\label{SSS:The-function-H}

We now present $H$ in a more digestible fashion.  
Combining all of the rational functions in \eqref{E:DefOfH(r)} yields
\begin{equation}\label{E:def-of-H}
H(r) = \frac{F_1(r)}{6250r^3(3r+1)^3(3r+4)^3(3r+7)^3} - 2r\log\bigg(\frac{3r+7}{3r+4} \bigg),
\end{equation}
where $F_1$ is a polynomial of degree 12 with integer coefficients, the full formula of which is included in Table \ref{table:Fj-coeffs} below.
The leading coefficient of $F_1$ (written $c_{1,12}$ in the table) is $246\,037\,500 = 2^2 3^9 5^5$.
Note that the denominator of the rational function in \eqref{E:def-of-H} is also a polynomial of degree 12, with leading coefficient $123\,018\,750 = 2^1 3^9 5^5$.
Thus, 
\begin{equation*}
\lim_{r \to \infty} H(r) = 2-2 = 0.
\end{equation*}

Now differentiate $H$ to obtain
\begin{equation}\label{E:DefOfH'}
H'(r) = \frac{F_2(r)}{3125r^4(3r+1)^4(3r+4)^4(3r+7)^4} - 2\log\bigg(\frac{3r+7}{3r+4} \bigg),
\end{equation}
where $F_2$ is a polynomial of degree 15 with integer coefficients (the full formula is also included in Table \ref{table:Fj-coeffs}).  Since the denominator polynomial has degree 16, it follows that 
\begin{equation*}
\lim_{r \to \infty} H'(r) = 0.
\end{equation*}

Finally, differentiate $H'$ to obtain
\begin{equation}\label{E:DefOfH''}
H''(r) = \frac{F_3(r)}{3125r^5(3r+1)^5(3r+4)^5(3r+7)^5},
\end{equation}
where $F_3$ is a polynomial of degree 16 with integer coefficients, the exact expression of which is included in Table \ref{table:Fj-coeffs}. 
Since the denominator polynomial has degree 20, it follows that
\begin{equation*}
\lim_{r \to \infty} H''(r) = 0.
\end{equation*}

The following statements are easily verified from Table \ref{table:Fj-coeffs} below:
\begin{enumerate}

\item The coefficients of the $1,r,\cdots,r^{7}$ terms in $F_3$ are negative.

\item The coefficients of the $r^8, r^9, \cdots, r^{16}$ terms in $F_3$ are positive.

\item Descartes' Rule of Signs thus guarantees that $F_3$ has a unique positive real root.  
Since the denominator of $H''(r)$ is positive for $r>0$, we see that $H''(r)$ changes sign exactly once for $r>0$.

\item Evaluation of the rational function $H''$ shows this sign change occurs inside the interval $\big(\frac{1}{3},\frac{2}{3}\big)$:
\begin{equation*}
H''\left(\tfrac{1}{3}\right) = -\frac{437\,616\,243}{25\,600\,000}, \qquad \mathrm{and} \qquad H''\left(\tfrac{2}{3}\right) = \frac{49\,618}{2\,278\,125}.
\end{equation*}

\item We conclude that $H''>0$ on the interval $\big(\frac{2}{3},\infty\big)$.

\item  Since $H''>0$ on $\big(\frac{2}{3},\infty\big)$, $H'$ strictly increases on this interval.  
Since $\lim\limits_{r \to \infty} H'(r) = 0$, we conclude that $H'<0$ on $\big(\frac{2}{3},\infty\big)$.

\item Since $H'<0$ on $\big(\frac{2}{3},\infty\big)$, $H$ strictly decreases on this interval.  Since $\lim\limits_{r \to \infty} H(r) = 0$, we conclude that $H>0$ on $\big(\frac{2}{3},\infty\big)$.
\end{enumerate}
Since $H(r)$ is positive for $r > \frac{2}{3}$, the expression in \eqref{E:New-lower-estimate-a} is positive for this range of $r$. 
This in turn shows that \eqref{E:AppQuant2} = \eqref{E:AppQuant1} is positive for the same range of $r$.
Thus by \eqref{E:AppQuant0},
\begin{align*}
\Phi\big(r,\tfrac{2}{3} \big) = \sum_{j=1}^\infty f_r(j) > 1 + H(r) > 1
\end{align*}
for $r>\tfrac{2}{3}$, completing the proof.
\end{proof} 

%%%%%%%%%%%%%%%%%%%%%%%%%%%%%%%%%%%%%%%%%%%%%%%%%%%%%%%%%%%%%%%%%%%%%%%%%
%%%% TABLE 1

\begin{table}[ht]
\centering
\begin{tabular}{ |p{.4cm}||p{2.6 cm}|p{3.0 cm}|p{3.6 cm}|  }
\hline
%\multicolumn{4}{|c|}{Table of coefficients} \\
%\hline
\hfill  $n$ & \hfill $c_{1,n}$ & \hfill $c_{2,n}$ & \hfill $c_{3,n}$ \\
\hline
\hline
\hfill $0$ & \hfill $-702\,464$ & \hfill $29\,503\,488$ & \hfill $-3\,304\,390\,656$ \\
\hline
\hfill $1$ & \hfill $-8\,805\,888$ & \hfill $493\,129\,728$ & \hfill $-69\,038\,161\,920$  \\
\hline
\hfill $2$ & \hfill $-44\,924\,544$ & \hfill $3\,546\,063\,360$ & \hfill $-640\,689\,315\,840$  \\
\hline
\hfill $3$ & \hfill $-258\,414\,880$ & \hfill $14\,430\,286\,080$ & \hfill $-3\,491\,968\,112\,640$  \\
\hline
\hfill $4$ & \hfill $1\,018\,286\,832$ & \hfill $77\,896\,979\,088$ & \hfill $-12\,471\,183\,325\,440$  \\
\hline
\hfill $5$ & \hfill $4\,962\,569\,148$ & \hfill $110\,838\,411\,360$ & \hfill $-48\,684\,386\,314\,944$  \\
\hline
\hfill $6$ & \hfill $11\,832\,384\,015$ & \hfill $-17\,706\,703\,248$ & \hfill $-111\,582\,268\,515\,360$  \\
\hline
\hfill $7$ & \hfill $23\,240\,472\,534$ & \hfill $244\,982\,773\,080$ & \hfill $-78\,421\,336\,513\,920$  \\
\hline
\hfill $8$ & \hfill $29\,834\,360\,478$ & \hfill  $1\,512\,143\,688\,033$ & \hfill $148\,629\,164\,640\,120$  \\
\hline
\hfill $9$ & \hfill $23\,154\,232\,644$ & \hfill $2\,940\,847\,647\,885$ & \hfill $378\,180\,897\,173\,910$  \\
\hline
\hfill $10$ & \hfill $10\,449\,759\,375$ & \hfill $3\,231\,415\,617\,165$ & \hfill $377\,142\,473\,066\,319$  \\
\hline
\hfill $11$ & \hfill $2\,501\,381\,250$ & \hfill $2\,264\,445\,221\,688$ & \hfill $224\,889\,469\,312\,590$  \\
\hline
\hfill $12$ & \hfill $246\,037\,500$ & \hfill $1\,025\,079\,243\,543$ & \hfill $92\,232\,089\,533\,215$  \\
\hline
\hfill $13$ & \hfill $0$ & \hfill $290\,262\,740\,625$ & \hfill $25\,224\,576\,030\,090$  \\
\hline
\hfill $14$ & \hfill $0$ & \hfill $47\,054\,671\,875$ & \hfill $3\,414\,213\,475\,245$  \\
\hline
\hfill $15$ & \hfill $0$ & \hfill $3\,321\,506\,250$ & \hfill $66\,996\,641\,106$  \\
\hline
\hfill $16$ & \hfill $0$ & \hfill $0$ & \hfill $14\,946\,778\,125$  \\
\hline
\end{tabular}
\caption{Exact values of coefficients of polynomials $F_j(r) = \sum c_{j,n} r^n$ in Section \ref{SSS:The-function-H}.}
\label{table:Fj-coeffs}
\end{table}

%%%%%%%%%%%%%%%%%%%%%%%%%%%%%%%%%%%%%%%%%%%%%%%%%%%%%%%%%%%%%%%%%%%%%%%%%
%%%%%%%%%%%%%%%%%%%%%%%%%%%%%%%%%%%%%%%%%%%%%%%%%%%%%%%%%%%

\subsection{Consequences for Leray transform}\label{SS:proof-of-preferred-Leray-norm}
We now prove that for $\gamma \neq 2$, the $L^2(M_\gamma,\nu)$ norm of $\bm{L}_k$ is strictly increasing in $k$.

\begin{theorem}\label{T:preferred-symbol-mono-increase}
Let $\gamma \neq2$.
Then $k \mapsto C_\nu(\gamma,k)$ is a strictly increasing function on the non-negative integers.
Thus,
\begin{equation*}
\norm{\bm{L}_k}_{L^2(M_\gamma,\nu)} < \norm{\bm{L}_{k+1}}_{L^2(M_\gamma,\nu)}.
\end{equation*}
\end{theorem}
\begin{proof}
We have just seen in Theorem \ref{T:EdhShe-series} that $\Phi\big(r,\tfrac{2}{3} \big) > 1$ for $r>\frac{2}{3}$. 
Lemma \ref{L:Ratio-into-series} now applies, saying the preferred symbol function $k \mapsto C_\nu(\gamma,k)$ is strictly increasing on the non-negative integers.
Thus, recalling \eqref{E:def-of-preferred-symbol} and \eqref{E:norm_Lk-sqrt}, we have
\begin{equation}\label{E:proof-of-mono-inc-L}
\norm{\bm{L}_k}_{L^2(M_\gamma,\nu)} = \sqrt{C_\nu(\gamma,k)} < \sqrt{C_\nu(\gamma,k+1)} = \norm{\bm{L}_{k+1}}_{L^2(M_\gamma,\nu)},
\end{equation}
completing the proof.
\end{proof}

We now easily deduce the $L^2(M_\gamma,\nu)$ norm of the full Leray transform.

\begin{theorem}\label{T:Leray=preferred-norm}
The norm of the Leray transform on $L^2(M_\gamma,\nu)$ is given by
\begin{equation*}
\norm{\bm{L}}_{L^2(M_\gamma,\nu)} = \sqrt{\frac{\gamma}{2\sqrt{\gamma-1}}}.
\end{equation*}
\end{theorem}
\begin{proof}[Proof of Theorem \ref{T:Leray=preferred-norm}]
The strictly increasing behavior seen in \eqref{E:proof-of-mono-inc-L} combines with \eqref{E:Leray-norm-as-sup} and Proposition \ref{P:HF-norm} to give
\[
\norm{\bm{L}}_{L^2(M_\gamma,\nu)} = \lim_{k\to\infty} \norm{\bm{L}_k} = \sqrt{\frac{\gamma}{2\sqrt{\gamma-1}}},
\]
completing the proof.
\end{proof}

\appendix

\section{Proof of the estimate on $S(r,Q_r)$}\label{A:Appendex-UpperBound}

In this appendix we prove \eqref{E:estimate-on-S(r,Qr)}, the crucial upper bound on $S(r,Q_r)$ that was used in our proof of Theorem \ref{T:EdhShe-series}.
Recall the definition of the rational function 
\begin{equation*}
S(r,x) = \frac{81r(9x^2-3x-9r^2-2)}{(3r+3x+1)^3(3r+3x-2)^3}.
\end{equation*}
It was shown in Section \ref{SSS:Analysis-on-S} that for $r>\frac{2}{3}$, $x \mapsto S(r,x)$ has a single local extrema (a maximum) at $x=Q_r$, where $Q_r$ is the unique positive root of the polynomial
\begin{equation}\label{E:appendix-p_r}
p_r(x) = (-4+39r-36r^2+162r^3) + (18 + 18 r + 216 r^2) x + (54 - 54 r) x^2 - 108 x^3.
\end{equation}
It was previously demonstrated that $Q_r \in \big(\tfrac{3r}{2}+\tfrac{1}{6},\tfrac{3r}{2}+\tfrac{1}{3}\big)$, but here sharper precision is needed; see Remark \ref{R:expansion-of-Q-at-infty}.

\subsection{A sharper estimate on $Q_r$}
Let us define
\begin{subequations}\label{E:new-upper-and-lower-bounds-Qr}
\begin{align}
m(r) &= \frac{3r}{2}+\frac{1}{6} + \frac{2}{25r} -\frac{21}{3125r^3} \label{E:def-of-lower-bound-m} \\
M(r) &= \frac{3r}{2}+\frac{1}{6} + \frac{2}{25r}. \label{E:def-of-upper-bound-M}
\end{align}
\end{subequations}
(These functions are truncated Taylor expansions of $Q_r$ at $\infty$; see Remark \ref{R:expansion-of-Q-at-infty}.)

\begin{lemma}\label{L:Qr-bound-appendix}
Let $m(r)$ and $M(r)$ be as above. The following inequality holds
\begin{equation*}
m(r) < Q_r < M(r).
\end{equation*}
\end{lemma}
\begin{proof}
This can be seen from direct evaluation.
Indeed, 
\begin{equation*}
p_r(m(r)) = \frac{81 (12348 + 
125 r^2 (1515625 r^4 + 21000 r^2 -5292))}{5^{15}\, r^9},
\end{equation*}
and it is easily checked that $1515625 r^4 + 21000 r^2 -5292 > 0$ for $r>\frac{2}{3}$.
Since all other terms in the formula are clearly positive, we conclude $p_r(m(r)) > 0$.

Now calculate
\begin{equation*}
p_r(M(r)) = -\frac{81 \left(36 + 875 r^2 \right)}{5^6\, r^3},
\end{equation*}
which is clearly negative for $r>\frac{2}{3}$.
Since $Q_r$ is the unique positive root of $p_r$, we conclude it must lie in the interval $(m(r),M(r))$.
\end{proof}

With these improved bounds on $Q_r$ we are ready to prove the desired estimate:

\begin{lemma}\label{L:estmate-S(r,Qr)}
The following estimate holds for $r>\frac{2}{3}$ 
\begin{equation}\label{E:estmate-S(r,Qr)}
S(r,Q_r) < \frac{16}{3125 r^3}.
\end{equation} 
\end{lemma}
\begin{proof}
First define the following two variable polynomial
\begin{equation}\label{E:def-W-function}
W(r,Q) = 2^4(3r+3Q+1)^3(3r+3Q-2)^3 - 3^4 5^5 r^4 (9Q^2-3Q-9r^2-2).
\end{equation}
This function is closely tied to the inequality in \eqref{E:estmate-S(r,Qr)}.
Indeed,
\begin{equation*}
W(r,Q_r) > 0 \qquad \Longleftrightarrow \qquad S(r,Q_r) < \frac{16}{3125 r^3}.
\end{equation*}

Now expand \eqref{E:def-W-function} and collect like-terms to express $W$ as
\begin{equation*}
W(r,Q) = \sum\nolimits_{j,k}  a_{j,k}\, r^j Q^k.
\end{equation*}
Let $\ch$ be the Heaviside step function (the indicator function of the positive real numbers) and define two closely related polynomials with {\em positive} integer coefficients:
\begin{subequations}
\begin{align}
U(r,Q) &= \sum\nolimits_{j,k}  \ch(a_{j,k})\, a_{j,k} \, r^j Q^k \label{E:def-of-U}\\
V(r,Q) &= -\sum\nolimits_{j,k} \ch(-a_{j,k})\, a_{j,k} \, r^j Q^k. \label{E:def-of-V}
\end{align}
\end{subequations}
Clearly, we have that
\begin{equation*}
W(r,Q) = U(r,Q) - V(r,Q).
\end{equation*}

Since both $U$ and $V$ have only positive coefficients in their expansions, Lemma \ref{L:Qr-bound-appendix} implies
\begin{align*}
U(r,Q_r) > U(r,m(r)), \qquad \qquad V(r,M(r)) > V(r,Q_r).
\end{align*}
Our goal will be to prove that
\begin{equation*}
U(r,m(r)) > V(r,M(r)),
\end{equation*}
which will imply $W(r,Q_r)>0$ and thereby give the result.

\subsection{Analysis of the polynomials $U(r,Q)$ and $V(r,Q)$}

The polynomials $U$ and $V$ are obtained from $W$ by expanding \eqref{E:def-W-function} and separating the monomial terms by the signs of their coefficients. 
After collecting terms in this way, we may re-write \eqref{E:def-of-U} and \eqref{E:def-of-V} as polynomials in the $Q$ variable:
\begin{equation}\label{E:U-and-V-expanded-in-Q}
U(r,Q) = \sum_{k=0}^6 u_k(r) Q^k, \qquad \qquad V(r,Q) = \sum_{k=0}^5 v_k(r) Q^k,
\end{equation}
where the coefficient functions $u_k(r)$ and $v_k(r)$ are given in Table \ref{table:Polynomial-coeffs-U-and-V}.

%%%%%%%%%%%%%%%%%%%%%%%%%%%%%%%%%%%%%%%%%%%%%%%%%%%%%%%%%
%%%%%%%%%%%%%%%%%%%%% TABLE 2 %%%%%%%%%%%%%%%%%%%%%%%%%%%
\begin{table}[ht]
\centering
\renewcommand{\arraystretch}{1.2}
\begin{tabular}{ |p{.4 cm}||p{6.9 cm}|p{5.7 cm}|  }
\hline
%\multicolumn{4}{|c|}{Table of coefficients} \\
%\hline
\centering{$k$}  & \hfill $u_k(r)$ & \hfill $v_k(r)$  \\
\hline
\hline
\centering{$0$} & \hfill $2289789r^6 +502362r^4 +4752r^3 +864r^2$ & \hfill $11664r^5 +576r +128$  \\
\hline
\centering{$1$} & \hfill $69984r^5 +701055r^4 +14256r^2 +1728r$ & \hfill $15552r^3 + 576$   \\
\hline
\centering{$2$} & \hfill $14256r +864$ & \hfill $2103165r^4 +116640r^3  +23328r^2 $   \\
\hline
\centering{$3$} & \hfill $233280r^3 +4752$ & \hfill $116640r^2 +15552r$   \\
\hline
\centering{$4$} & \hfill $174960 r^2$ & \hfill $58320r +3888$   \\
\hline
\centering{$5$} & \hfill $69984 r$ & \hfill $11664$   \\
\hline
\centering{$6$} & \hfill $11664$ & \hfill $0$   \\
\hline
\end{tabular}
\renewcommand{\arraystretch}{1}
\caption{Coefficients polynomials in $U(r,Q)$ and $V(r,Q)$}
\label{table:Polynomial-coeffs-U-and-V}
\end{table}
%%%%%%%%%%%%%%%%%%%%%%%%%%%%%%%%%%%%%%%%%%%%%%%%%%%%%%%%%
%%%%%%%%%%%%%%%%%%%%%%%%%%%%%%%%%%%%%%%%%%%%%%%%%%%%%%%%%

Now insert $m(r)$ and $M(r)$ into \eqref{E:U-and-V-expanded-in-Q} and expand $U(r,m(r)) - V(r,M(r))$ out as a rational function of the form
\[
\sum b_n r^n.
\]
Using Table \ref{table:Polynomial-coeffs-U-and-V} together with degree considerations, we see that $b_n = 0$ for $n < -18$ and $n > 6$.
Some of the remaining $b_j$, including $b_5$ and $b_6$, are also equal to zero; see Table \ref{table:exact-of-coeffs-of-P}, noting that $\beta_n = b_{n-18}$.
But it is easily checked from \eqref{E:U-and-V-expanded-in-Q} and Table \ref{table:Polynomial-coeffs-U-and-V} that $b_{-18} \neq 0$, which leads to the definition of the related polynomial
\begin{align*}
P(r) &= r^{18}(U(r,m(r)) - V(r,M(r))) = \sum_{n=0}^{22} \beta_n r^n
\end{align*}
The exact values of the $\beta_n$ are listed in Table \ref{table:exact-of-coeffs-of-P}.
But we can explain the next step of the argument just by considering the signs of the $\beta_n$ in Table \ref{table:signs-of-coeffs-of-P}.

%%%%%%%%%%%%%%%%%%%%%%%%%%%%%%%%%%%%%%%%%%%%%%%%%%%%%%%%%
%%%%%%%%%%%%%%%%%%%%% TABLE 3 %%%%%%%%%%%%%%%%%%%%%%%%%%%
\begin{table}[ht]
\centering
\begin{tblr}{ |p{.22 cm}|p{.22 cm}|p{.22 cm}|p{.22 cm}|p{.22 cm}|p{.22 cm}|p{.22 cm}|p{.22 cm}|p{.22 cm}|p{.22 cm}|p{.22 cm}|p{.22 cm}|p{.22 cm}|p{.22 cm}|p{.22 cm}|p{.22 cm}|p{.22 cm}|p{.22 cm}|p{.22 cm}|p{.22 cm}|p{.22 cm}|p{.22 cm}|p{.22 cm}|  }
\hline
%\multicolumn{4}{|c|}{Table of coefficients} \\
$0$  &  $1$  & $2$  & $3$  & $4$  & $5$  & $6$  & $7$  & $8$  & $9$  & $\! 10$  & $\! 11$  & $\! 12$  & $\! 13$  & $\! 14$  & $\! 15$  & $\! 16$  & $\! 17$  & $\! 18$  & $\! 19$  & $\! 20$  & $\! 21$ & $\! 22$  \\
\hline
 $\! +$  & $0$  & \centering{$\! -$}  & \centering{$\! -$}  & \centering{$\! +$}  & \centering{$\! +$}  & \centering{$\! +$}  & \centering{$\! -$}  & \centering{$\! -$}  & \centering{$-$}  & \centering{$-$}  & \centering{$+$}  & \centering{$+$}  & \centering{$+$}  & \centering{$-$}  & \centering{$-$}  & \centering{$-$}  & \centering{$-$}  & \centering{$-$}  & \centering{$-$}  & \centering{$-$}  & \centering{$0$} & \centering{$+$}  \\
\hline
\end{tblr}
\caption{Signs of coefficients $\beta_n$ in $P(r)$}
\label{table:signs-of-coeffs-of-P}
\end{table}
%%%%%%%%%%%%%%%%%%%%%%%%%%%%%%%%%%%%%%%%%%%%%%%%%%%%%%%%%
%%%%%%%%%%%%%%%%%%%%%%%%%%%%%%%%%%%%%%%%%%%%%%%%%%%%%%%%%

Table~\ref{table:signs-of-coeffs-of-P} shows that the coefficients of $P(r)$ change signs six times when its terms are listed in ascending order; 
Descartes' Rule of Signs thus says that the number of positive real roots of $P$ is either $0$, $2$, $4$ or $6$.
This is not precise enough for our purposes, but we can cut down the number of sign changes by repeatedly differentiating.
After fourteen derivatives, Table~\ref{table:signs-of-coeffs-of-P} shows the resulting polynomial has {\em exactly one} coefficient sign change when the coefficients are listed in ascending order.
Consequently, $P^{(14)}$ has a unique positive real root.

The remaining computations are nothing more than differentiation and direct polynomial evaluation: $P(r)$ and its first thirteen derivatives are evaluated at $r=\tfrac{2}{3}$, while $P^{(14)}(r)$ is evaluated at $r=0$ and $\tfrac{2}{3}$.
These calculations are tedious by hand, but they can be worked out in totality using Table~\ref{table:exact-of-coeffs-of-P}. 
(See also the supplementary Mathematica notebook found by following the link given in \cite{Edholm_Leray_nb}.)

Evaluation shows that the unique root of $P^{(14)}(r)$ lies in the interval $(0,\tfrac{2}{3})$. Indeed,
\begin{align*}
&P^{(14)}(0) = -\frac{2\,159\,106\,379\,702\,272}{5^7} \approx -2.76366 \cdot 10^{10}, \\
&P^{(14)}\Big(\frac{2}{3}\Big) = \frac{18\,441\,535\,745\,869\,667\,168\,145\,408}{5^7} \approx 2.36052 \cdot 10^{20}.
\end{align*}
We thus conclude that 
\[
P^{(14)}(r) > 0, \qquad r>\frac{2}{3}.
\]
It can also be seen by direct evaluation that for integers $0 \le n \le 13$,
\[
P^{(n)}\Big(\frac{2}{3}\Big)>0.
\]
Decreasing the number of derivatives one step at a time, we see that for integers $0 \le n \le 13$, it also holds that
\[
P^{(n)}(r)>0, \qquad r > \frac{2}{3}.
\]

In particular, $P(r)>0$ for $r>\tfrac{2}{3}$, so the difference $U(r,m(r))-V(r,M(r))>0$.
This implies $U(r,Q_r)-V(r,Q_r) = W(r,Q_r)>0$, thereby proving \eqref{E:estmate-S(r,Qr)}.
\end{proof}

%%%%%%%%%%%%%%%%%%%%%%%%%%%%%%%%%%%%%%%%%%%%%%%%%%%%%%%%%
%%%%%%%%%%%%%%%%%%%%% TABLE 4 %%%%%%%%%%%%%%%%%%%%%%%%%%%
\begin{table}[ht]
\centering
\renewcommand{\arraystretch}{1}
\begin{tblr}{ |p{.3 cm}|p{3 cm}||p{.4 cm}|p{3 cm}||p{.4 cm}|p{3 cm}|  }
\hline
%\multicolumn{4}{|c|}{Table of coefficients} \\
%\hline
\centering{$n$}  & \centering{$\beta_n$} & \centering{$7$}  & \centering{$-\dfrac{162\,030\,456}{5^{17}}$} & \centering{$15$}  & \centering{$-\dfrac{3\,195\,801}{5^6}$}  \\
\hline
\centering{$0$} & \centering{ $\dfrac{1\,000\,376\,035\,344}{5^{30}}$} & \centering{$8$} & \centering{$-\dfrac{3\,421\,928\,916}{5^{17}}$} & \centering{$16$} & \centering{$-\dfrac{2\,065\,794\,597}{5^9}$}  \\
\hline
\centering{$1$} & \centering{$0$} & \centering{$9$} & \centering{$-\dfrac{84\,873\,096}{5^{14}}$} & \centering{$17$} & \centering{$-\dfrac{91\,854}{5^2}$}  \\
\hline
\centering{$2$} & \centering{$-\dfrac{857\,465\,173\,152}{5^{27}}$} & \centering{$10$} & \centering{$-\dfrac{922\,948\,992}{5^{15}}$} & \centering{$18$} & \centering{$-\dfrac{629\,807\,157}{2^2 5^6}$}  \\
\hline
\centering{$3$} & \centering{$-\dfrac{47\,636\,954\,064}{5^{25}}$} & \centering{$11$} &  \centering{$\dfrac{17\,635\,968}{5^{11}}$} & \centering{$19$} & \centering{$ -\dfrac{12\,267\,612}{5^4}$}  \\
\hline
\centering{$4$} & \centering{$\dfrac{163\,326\,699\,648}{5^{24}}$} & \centering{$12$} & \centering{$\dfrac{657\,460\,071}{5^{12}}$} & \centering{$20$} & \centering{$-\dfrac{71\,827\,641}{2^2 5^4}$}  \\
\hline
\centering{$5$} & \centering{$\dfrac{6\,805\,279\,152}{5^{21}}$} & \centering{$13$} &  \centering{$\dfrac{10\,471\,356}{5^9}$} & \centering{$21$} & \centering{$0$}  \\
\hline
\centering{$6$} & \centering{$\dfrac{9\,694\,822\,284}{5^{20}}$} & \centering{$14$} & \centering{$-\dfrac{619\,164}{5^9}$} & \centering{$22$} & \centering{$\dfrac{455\,625}{2}$}  \\
\hline
\end{tblr}
\renewcommand{\arraystretch}{1}
\caption{Exact values of coefficients in $P(r) = \sum \beta_n r^n$}
\label{table:exact-of-coeffs-of-P}
\end{table}
%%%%%%%%%%%%%%%%%%%%%%%%%%%%%%%%%%%%%%%%%%%%%%%%%%%%%%%%%
%%%%%%%%%%%%%%%%%%%%%%%%%%%%%%%%%%%%%%%%%%%%%%%%%%%%%%%%%

\begin{remark}\label{R:expansion-of-Q-at-infty}
Applying the cubic formula to the polynomial $p_r(x)$ given in \eqref{E:appendix-p_r}, we have
\begin{align*}
&Q_r = \frac{1}{6\alpha_r}\left( 3+25r^2 + (1-r)\alpha_r + \alpha_r^2 \right), \qquad \mathrm{where}  \\
&\alpha_r = \left(125r^3 +36r - 3\sqrt{375r^4+69r^2-3} \right)^{1/3}. \label{E:def-of-alpha_r}
\end{align*}

The bounds $M(r)$ and $m(r)$ in Lemma \ref{L:Qr-bound-appendix} come from the Taylor expansion of $Q_r$ at $\infty$:
\begin{equation*}
Q_r = \frac{3r}{2} + \frac{1}{6} + \frac{3}{25r} - \frac{21}{3125r^3} + O\Big(\frac{1}{r^4}\Big).
\end{equation*}
It should be emphasized that the negative degree terms in $M(r)$ and $m(r)$ are crucial to the proof of Lemma \ref{L:estmate-S(r,Qr)}; simply using affine functions to bound $Q_r$ is not sufficient to prove the estimate on $S(r,Q_r)$.

Similarly, the bound on $S(r,Q_r)$ also comes from its the Taylor expansion at $\infty$:
\begin{equation*}
S(r,Q_r) = \frac{16}{3125r^3} + O\Big(\frac{1}{r^5}\Big)
\end{equation*}
This method of generating candidates for sufficiently sharp bounds is likely to have application for many other polygamma inequalities, and more generally, in other problems in which the Euler-Maclaurin formula is utilized.
\hfill $\lozenge$
\end{remark}

\bibliographystyle{acm}
\bibliography{EdhShe2024a}

\begin{thebibliography}{10}

\bibitem{AbrSteBook}
{\sc {Abramowitz}, M., and {Stegun}, I.~A.}
\newblock {\em Handbook of Mathematical Functions with Formulas, Graphs, and Mathematical Tables}, ninth {D}over printing, tenth {GPO} printing~ed.
\newblock Dover, New York City, 1964.

\bibitem{Alzer1997}
{\sc Alzer, H.}
\newblock On some inequalities for the gamma and psi functions.
\newblock {\em Math. Comp. 66}, 217 (1997), 373--389.

\bibitem{Alzer1998}
{\sc Alzer, H.}
\newblock Inequalities for the gamma and polygamma functions.
\newblock {\em Abh. Math. Sem. Univ. Hamburg 68\/} (1998), 363--372.

\bibitem{AndPasSigBook04}
{\sc Andersson, M., Passare, M., and Sigurdsson, R.}
\newblock {\em Complex convexity and analytic functionals}, vol.~225 of {\em Progress in Mathematics}.
\newblock Birkh\"auser Verlag, Basel, 2004.

\bibitem{Bar16}
{\sc Barrett, D.~E.}
\newblock Holomorphic projection and duality for domains in complex projective space.
\newblock {\em Trans. Amer. Math. Soc. 368}, 2 (2016), 827--850.

\bibitem{BarEdh20}
{\sc Barrett, D.~E., and Edholm, L.~D.}
\newblock The {L}eray transform: {F}actorization, dual {CR} structures, and model hypersurfaces in {CP}2.
\newblock {\em Adv. Math. 364\/} (2020), 107012.

\bibitem{BarEdh22}
{\sc Barrett, D.~E., and Edholm, L.~D.}
\newblock High frequency behavior of the {L}eray transform: model hypersurfaces and projective duality.
\newblock {\em (to appear) Indiana Univ. Math. J.\/} (2024).

\bibitem{BarLan09}
{\sc Barrett, D.~E., and Lanzani, L.}
\newblock The spectrum of the {L}eray transform for convex {R}einhardt domains in {$\mathbb{C}^2$}.
\newblock {\em J. Funct. Anal. 257}, 9 (2009), 2780--2819.

\bibitem{Bol05}
{\sc Bolt, M.}
\newblock A geometric characterization: complex ellipsoids and the {B}ochner-{M}artinelli kernel.
\newblock {\em Illinois J. Math. 49}, 3 (2005), 811--826.

\bibitem{Edholm_Leray_nb}
{\sc {Edholm,\,\,L.\,\,D.}}
\newblock Leray measures 2024 {M}athematica notebook.
\newblock \url{https://github.com/ledholm/Leray-measures-2024-Mathematica-nb/releases/tag/v1.0}, 2024.

\bibitem{GKPBook}
{\sc Graham, R.~L., Knuth, D.~E., and Patashnik, O.}
\newblock {\em Concrete mathematics}, second~ed.
\newblock Addison-Wesley Publishing Company, Reading, MA, 1994.
\newblock A foundation for computer science.

\bibitem{GuoGuoQi_2010}
{\sc Guo, B.-N., Guo, S., and Qi, F.}
\newblock Complete monotonicity of some functions involving polygamma functions.
\newblock {\em J. Comput. Appl. Math. 233}, 9 (2010), 2149--2160.

\bibitem{GuoQi_2012_AAM}
{\sc Guo, B.-N., and Qi, F.}
\newblock Necessary and sufficient conditions for functions involving the tri- and tetra-gamma functions to be completely monotonic.
\newblock {\em Adv. in Appl. Math. 44}, 1 (2010), 71--83.

\bibitem{GuoQi_2013_PAMS}
{\sc Guo, B.-N., and Qi, F.}
\newblock Refinements of lower bounds for polygamma functions.
\newblock {\em Proc. Amer. Math. Soc. 141}, 3 (2013), 1007--1015.

\bibitem{GuoQiZhao_2012}
{\sc Guo, B.-N., Qi, F., and Zhao, J.-L.}
\newblock Complete monotonicity of two functions involving the tri- and tetra-gamma functions.
\newblock {\em Period. Math. Hungar. 65}, 1 (2012), 147--155.

\bibitem{GuoQiSri_2012}
{\sc Guo, S., Qi, F., and Srivastava, H.~M.}
\newblock A class of logarithmically completely monotonic functions related to the gamma function with applications.
\newblock {\em Integral Transforms Spec. Funct. 23}, 8 (2012), 557--566.

\bibitem{LanSte13}
{\sc Lanzani, L., and Stein, E.~M.}
\newblock Cauchy-type integrals in several complex variables.
\newblock {\em Bull. Math. Sci. 3}, 2 (2013), 241--285.

\bibitem{LanSte14}
{\sc Lanzani, L., and Stein, E.~M.}
\newblock The {C}auchy integral in {$\mathbb{C}^n$} for domains with minimal smoothness.
\newblock {\em Adv. Math. 264\/} (2014), 776--830.

\bibitem{LanSte17c}
{\sc Lanzani, L., and Stein, E.~M.}
\newblock The role of an integration identity in the analysis of the {C}auchy-{L}eray transform.
\newblock {\em Sci. China Math. 60}, 11 (2017), 1923--1936.

\bibitem{LanSte19}
{\sc Lanzani, L., and Stein, E.~M.}
\newblock The {C}auchy-{L}eray integral: counterexamples to the {$L^p$}-theory.
\newblock {\em Indiana Univ. Math. J. 68}, 5 (2019), 1609--1621.

\bibitem{Lee1986}
{\sc Lee, J.~M.}
\newblock The {F}efferman metric and pseudo-{H}ermitian invariants.
\newblock {\em Trans. Amer. Math. Soc. 296}, 1 (1986), 411--429.

\bibitem{Merkle_2005}
{\sc Merkle, M.}
\newblock Gurland's ratio for the gamma function.
\newblock {\em Comput. Math. Appl. 49}, 2-3 (2005), 389--406.

\bibitem{Range86}
{\sc Range, R.~M.}
\newblock {\em Holomorphic functions and integral representations in several complex variables}, vol.~108 of {\em Graduate Texts in Mathematics}.
\newblock Springer-Verlag, New York, 1986.

\bibitem{Wang2004}
{\sc Wang, X.}
\newblock A simple proof of {D}escartes's {R}ule of {S}igns.
\newblock {\em The American Mathematical Monthly 111}, 6 (2004), 525--526.

\bibitem{WidderBook}
{\sc Widder, D.~V.}
\newblock {\em The {L}aplace {T}ransform}.
\newblock Princeton Mathematical Series, vol. 6. Princeton University Press, Princeton, NJ, 1941.

\end{thebibliography}

\end{document}